\documentclass[reqno,11pt]{amsart}

\usepackage{color}
\usepackage{graphicx}
\usepackage[hidelinks]{hyperref}
\usepackage{latexsym,array}
\usepackage{amsfonts}
\usepackage{shadow}
\usepackage{amsbsy}
\usepackage{dsfont}
\usepackage{mathtools}

\newcommand{\id}{\mathcal{I}}

\newcommand{\lhol}{\mathcal{SH}_L}
\newcommand{\rhol}{\mathcal{SH}_R}
\newcommand{\intrin}{\mathcal{N}}

\newtheorem{Pa}{Paper}[section]
\newtheorem{theorem}[Pa]{{\bf Theorem}}
\newtheorem{lemma}[Pa]{{\bf Lemma}}
\newtheorem{definition}[Pa]{{\bf Definition}}

\newtheorem{Pn}[Pa]{{\bf Proposition}}

\newtheorem{Ex}[Pa]{{\bf Example}}

\theoremstyle{definition}
\newtheorem{remark}[Pa]{{\bf Remark}}

\def\C{\mathbb C}
\def\hh{\mathbb{H}}

\newcommand{\rr}{\mathbb{R}}
\renewcommand{\S}{\mathbb{S}}

\newcommand{\boundOP}{\mathcal{B}}
\newcommand{\closOP}{\mathcal{K}}

\renewcommand{\Re}{\mathrm{Re}}

\title[
The $H^\infty$ Functional Calculus ]
{The $H^\infty$ functional calculus based on the $S$-spectrum for quaternionic operators and for $n$-tuples of noncommuting operators} \oddsidemargin
0.2in \evensidemargin 0.2in \topmargin -0.5in \textwidth
15.5truecm \textheight 23truecm
\def\H{\mathbb H}

\def\R{\mathbb R}
\def\N{\mathbb N}
\def\C{\mathbb C}

\def\(s){\mathscr S(\R\times\R)}

\author[D. Alpay]{Daniel Alpay}
\address{(DA)  Department of Mathematics\\
Ben Gurion University of the Negev\\
Beer Sheva 84105 Israel}
\email{dany@math.bgu.ac.il}

\author[F. Colombo]{Fabrizio Colombo}
\address{(FC)
Politecnico di Milano\\Dipartimento di Matematica\\Via E. Bonardi, 9\\20133
Milano, Italy}
\email{fabrizio.colombo@polimi.it}

\author[T. Qian]{Tao Qian}
\address{(TQ)
     Department of Mathematics\\
      Faculty of Science and Technology
      \\ University of Macau, Taipa
      \\ Macao, China
}
\email{fsttq@umac.mo}

\author[I. Sabadini]{Irene Sabadini}
\address{(IS)
Politecnico di Milano\\Dipartimento di Matematica\\Via E. Bonardi, 9\\20133
Milano, Italy}
\email{irene.sabadini@polimi.it}

\thanks{
The authors thank
Macao Science and Technology Fund FDCT/098/2012/A3
and
Macao Science and Technology Fund FDCT/099/2014/A2.
}

\begin{document}
\maketitle

\begin{abstract}
In this paper we extend the $H^\infty$ functional calculus to quaternionic operators and to $n$-tuples of noncommuting operators using the theory of slice hyperholomorphic functions and the associated functional calculus, called $S$-functional calculus.
The $S$-functional calculus has two versions one for quaternionic-valued functions and one for Clifford algebra-valued functions
and can be considered the Riesz-Dunford functional calculus based on slice hyperholomorphicity because it shares with it the most important properties.
 The $S$-functional calculus is based on the notion of $S$-spectrum which, in the case of quaternionic normal operators on a Hilbert space, is also the notion of spectrum that appears in the quaternionic spectral theorem.
 The main purpose of this paper is to construct the $H^\infty$ functional calculus based on the notion of $S$-spectrum for both quaternionic operators and for $n$-tuples of noncommuting operators. We remark that the $H^\infty$ functional calculus for $(n+1)$-tuples of operators applies, in particular, to the Dirac operator.
\end{abstract}
\vskip 1cm
\section{Introduction}

The $H^\infty$-functional calculus is an extension of the Riesz-Dunford functional calculus for bounded operators, see \cite{ds}, to unbounded sectorial operators and  it has been introduced by A. McIntosh in \cite{McI1}, see also \cite{ADMcI}.
This calculus is connected  with pseudo-differential operators, with the Kato's square root problem,
 and with the study of evolution equations and, in particular,
 the characterization of maximal regularity and of the fractional powers of differential operators.
For an overview and more problems associated with this functional calculus see the paper \cite{LW},
the book \cite{Haase} and the references therein.
\\
\\
One of the main motivations to study quaternionic operators is the fact that
 they are important in the formulation of quantum mechanics. In fact, it was proved by
G. Birkhoff and  J. von Neumann \cite {BvN}, that there are essentially
 two possible ways to formulate quantum mechanics: using complex numbers or quaternions, see \cite{adler}.
\\
\\
The main purpose of this paper is to construct the $H^\infty$ functional calculus based on the notion of $S$-spectrum for quaternionic operators and for $n$-tuples of noncommuting operators.
 To do this, we replace the Riesz-Dunford functional calculus by the
 $S$-functional calculus which is the quaternionic version of the
 Riesz-Dunford functional calculus, see \cite{acgs, JGA, CLOSED,MR2752913}. The $S$-functional calculus is based on the notion of $S$-spectrum which, in the case of quaternionic normal operators on a Hilbert space, is also the notion of spectrum that appears in the quaternionic spectral theorem, see
 \cite{ack, acks2, spectcomp}.
 The $S$-functional calculus is defined for slice hyperholomorphic functions:
for the quaternionic version of the function theory see the books \cite{ACSBOOK,MR2752913,GSSb}
and for the Clifford algebra version see \cite{MR2752913}.
\\
\\
We begin by recalling
some results and definitions from the complex setting, following the paper \cite{McI1} and the book
 \cite{Haase}.
Let  $A$ be a linear operator on a complex Banach space $X$, with dense domain $\mathcal{D}(A)$ and dense range ${\rm Ran}(A)$.
Let $\omega\in [0,\pi)$. We say that $A$ is of type $\omega$ if its spectrum $\sigma(A)$ is contained in the sector
$$
S_\omega=\{z\in \mathbb{C}\ :\ |\arg(z)|\leq \omega \} \cup \{0\}
$$
and if there exists a positive constant $c_\mu$, for $\mu >\omega$, such that
$$
\|(A-z I)^{-1}\|\leq \frac{c_\mu}{|z|},
$$
 for all $z$ such that $|\arg (z)|\geq \mu$.
\\
For this class of operators, called sectorial operators,
 it is possible to construct a functional calculus using bounded holomorphic functions $g$ for which there exists two positive constants $\alpha$ and $c$ such that
\begin{equation}\label{gicres}
 |g(z)|\leq \frac{c|z|^\alpha}{1+|z|^{2\alpha}}\ \ \ {\rm for \ all }\ \ z\in S^0_\omega,
 \end{equation}
where $S^0_\omega$ is the interior of $S_\omega$.
The strategy is based on the Cauchy formula for holomorphic functions in which we replace the Cauchy kernel by the resolvent operator $R(\lambda, A)$. In the case $A$ is a bounded linear operator then the spectrum of $A$ is a bounded and nonempty set in the complex plane,
 so using a suitable contour $\gamma$, that surrounds the spectrum of $A$ it is possible to define the bounded linear operator
\begin{equation}\label{FCA1}
g(A)=\frac{1}{2\pi i}\int_\gamma (zI-A)^{-1}g(z)dz.
\end{equation}
The  integral (\ref{FCA1}) turns out to be convergent
 for sectorial operators, if we assume that estimate (\ref{gicres}) holds for the bounded holomorphic function $g$.
  We point out that the definition is well posed, because the integral does not depend on the contour $\gamma$ when $\gamma$ does not intersect the spectrum of $A$.
\\
Now we extend the above  calculus so that we can define
operators such as  $A^\lambda$ with $\lambda\in \mathbb{C}$ or
$\lambda^{2-\beta}A^\beta R(\lambda,A)^2$ with $\beta\in (0,2)$.
Using the functional calculus defined in (\ref{FCA1}) and the rational functional calculus
\begin{equation}\label{FCArational}
\varphi(A)=\Big(A(I+A^2)^{-1}\Big)^{k+1},\ \ k\in \mathbb{N},
\end{equation}
where
$$
\varphi(z)=\Big(\frac{z}{1+z^2}\Big)^{k+1},\ \ k\in \mathbb{N},
$$
we can define a more general functional calculus for sectorial operators given by
\begin{equation}\label{FCA2}
f(A)=(\varphi(A))^{-1}(f\varphi)(A)
\end{equation}
where $f$ is a holomorphic function on $S^0_\omega$ which satisfies bounds of the type
$$
|f(z)|\leq c(|z|^k+|z|^{-k}), \ \ {\rm for }\ \ c>0, \ k>0.
$$
The calculus defined in (\ref{FCA2}) is called the $H^\infty$ functional calculus and has been introduced in \cite{McI1}.
Note that, strictly speaking, we should write $\varphi_k$ instead of $\varphi$, but we omit the subscript for the sake of simplicity.
The definition above is well posed since  the operator $f(A)$ does not depend on the suitable rational function $\varphi$ that we choose.

Moreover, observe that the operator $(f\varphi)(A)$ can be defined using the functional calculus
(\ref{FCA1}) for the function $f\varphi$, where $\varphi(z)=(z(1+z^2)^{-1})^{k+1}$.
The operator $(f\varphi)(A)$ is bounded but $(\varphi(A))^{-1}$ is closed, so
the operator $f(A)$ defined in (\ref{FCA2}) is not necessarily bounded.
This calculus is very important because of Theorem \ref{convth} that we state in the sequel.
\\
\\
To explain how we can extend the $H^\infty$ functional calculus to the quaternionic setting we must make precise the notions of spectrum, of resolvent operator, of holomorphicity.
With the standard imaginary units $e_1$, $e_2$ obeying $e_1e_2+e_2e_1=0$, $e_1^2=e_2^2=-1$ and $e_3:=e_1e_2$, the algebra of quaternions $\mathbb{H}$ consists of elements of the form $q=x_0+x_1e_1+x_2e_2+x_3e_3$, for $x_\ell\in\mathbb{R}$, for $\ell=0,\ldots ,3$.
The real part, imaginary part and the square of the modulus of a quaternion are defined as
${\rm Re}\  q=x_0$, ${\rm Im}\  q= x_1e_1 + x_2e_2 + x_3e_3$,
$|q|^2=x_0^2+x_1^2+x_2^2+x_3^2$, respectively.
The conjugate $\bar q$ of the quaternion $q$ is defined by $\bar q={\rm Re }\ q-{\rm Im }\ q=x_0-x_1e_1 - x_2e_2 - x_3e_3$
 and it satisfies
$$
|q|^2=q\bar q=\bar q q.
$$
By $\mathbb{S}$ we denote the sphere of purely imaginary quaternions whose square is $-1$.
Every element $i\in \mathbb{S}$ works as an imaginary unit and with each $i\in \mathbb{S}$ we associate the complex plane
$\mathbb{C}_i=\{u+iv\ :\ u,v\in \mathbb{R}\}$ so that we have that $\mathbb{H}$ can be seen as the union of the complex planes $\mathbb{C}_i$ when $i$ varies in $\mathbb{S}$.
In this paper for the quaternionic setting, and in the Clifford algebra setting we use the notion
  of slice hyperholomorphicity, see Section \ref{SEc2}.
Regarding operators  we replace the classical notion of spectrum of an operator by the $S$-spectrum of a quaternionic operator
(resp. the $S$-spectrum of the $n$-tuples of operators) and the resolvent operator by the two $S$-resolvent operators which are slice hyperholomorphic functions operator--valued, see the book \cite{MR2752913}.
\\
 Precisely, we define the $S$-spectrum of the bounded quaternionic linear operator $T$ as
 $$
 \sigma_S(T)=\{s\in \mathbb{H} \ : \   T^2-2{\rm Re}(s)T+|s|^2\mathcal{I} \  \ {\rm is \ not\ invertible\ in} \ \mathcal{B}(V) \}
 $$
 where $\mathcal{B}(V)$ denotes the space of all bounded linear operators on a two-sided quaternionic Banach space $V$.
  In the case of bounded quaternionic linear operators the $S$-spectrum is a nonempty and compact set.
 The $S$-resolvent set $\rho_S(T)$ is defined by
$$
\rho_S(T)=\mathbb{H}\setminus\sigma_S(T).
$$
\noindent
For $s\in \rho_S(T)$ we define the left $S$-resolvent operator as
\begin{equation}\label{LEFTROPREW}
S_L^{-1}(s,T):=-(T^2-2s_0 T+|s|^2\mathcal{I})^{-1}(T-\overline{s}\mathcal{I}),
\end{equation}
and the right $S$-resolvent operator as
\begin{equation}\label{RAITPREW}
S_R^{-1}(s,T):=-(T-\overline{s}\mathcal{I})(T^2-2s_0 T+|s|^2\mathcal{I})^{-1}.
\end{equation}
Let $U\subset \mathbb{H}$ be a suitable domain that contains the $S$-spectrum of $T$.
We define the quaternionic functional calculus for left slice hyperholomorphic functions $f:U \to \mathbb{H}$ as
\begin{equation}\label{quatinteg311def}
f(T)={{1}\over{2\pi }} \int_{\partial (U\cap \mathbb{C}_i)} S_L^{-1} (s,T)\  ds_i \ f(s),
\end{equation}
where $ds_i=-i ds$;
for right slice hyperholomorphic functions, we define \begin{equation}\label{quatinteg311rightdef}
f(T)={{1}\over{2\pi }} \int_{\partial (U\cap \mathbb{C}_i)} \  f(s)\ ds_i \ S_R^{-1} (s,T).
\end{equation}
These definitions are well posed since the integrals depend neither on the open set $U$ nor on the complex plane
$\mathbb{C}_i$, for $i\in \mathbb{S}$.
\\
\\
In the following we will consider just right linear quaternionic operators (similar considerations can be done when we consider left linear operators).
In order to extend the $S$-functional calculus to closed operators it is necessary that the two $S$-resolvent operators are defined
not only on the domain of $T$ but they must be defined for all elements in the two-sided quaternionic Banach space $V$.
So for closed operators we define the $S$-resolvent set as
\[
\rho_S(T):= \{ s\in\H \ :\  (T^2 - 2\Re(s)T + |s|^2\id)^{-1}\in\boundOP(V)\},
\]
which we always suppose it to be nonempty,
and the $S$-spectrum of $T$ as
\[\sigma_S(T):=\H\setminus\rho_S(T).\]
For $s\in\rho_S(T)$ the  left $S$-resolvent operator is defined as
\begin{equation}
S_L^{-1}(s,T):= Q_s(T)\overline{s} -TQ_s(T)
\end{equation}
while the right $S$-resolvent operator remains as in (\ref{RAITPREW}), we simply rewrite it in terms of $Q_s(T)$ as:
\begin{equation}
S_R^{-1}(s,T):=-(T-\id \overline{s})Q_s(T)
\end{equation}
where
$
Q_s(T):=(T^2-2\Re(s)T +|s|^2\id )^{-1}
$
is called the pseudo-resolvent operator of $T$.

The quaternionic rational functions that we will use  are {\em intrinsic} rational slice hyperholomorphic functions.
 This class of functions is of fundamental importance and allows the definition
of  a rational functional calculus that includes the operators:
$$
\psi(T)= \Big(T(\mathcal{I}+T^2)^{-1}\Big)^{k+1}, \ \ k\in \mathbb{N}.
$$
Note that, also in this case, we write $\psi$ instead of $\psi_k$.
We extend the $S$-functional calculus to sectorial operators in the quaternionic setting, and then we use the classical regularizing procedure to define
 the extended functional calculus for slice hyperholomorphic functions $f$ with suitable growth conditions
$$
f(T):=(\psi(T))^{-1}(\psi f)(T),
$$
where the operator $(\psi f)(T)$ is defined using the $S$-functional calculus for sectorial operators,
and $\psi(T)$ is defined by the rational quaternionic functional calculus.
The definition does not depend on  the suitable rational function $\psi$ that we choose.
\\
\\
Examples of operators to which this calculus applies are:
\begin{itemize}
 \item[(i)]The Cauchy-Fueter operator (and its variations)
 $$
 \frac{\partial }{\partial x_0}+\sum_{j=1}^3  e_j\frac{\partial }{\partial x_j}.
 $$
\item[(ii)]
 Quaternionic operators appearing in quaternionic quantum mechanics such as the Hamiltonian, see \cite{adler}.
\item[(iii)]
 The global operator (see \cite{Global}) that annihilates slice hyperholomorphic  functions:
 \\
$$
|\underline{q}|^2\frac{\partial  }{\partial x_0}  \ +  \  \underline{q}  \  \sum_{j=1}^3  x_j\frac{\partial }{\partial x_j}, \ \ \ {\rm where} \ \ \ \underline{q}=x_1e_1 + x_2e_2 + x_3e_3.
$$

\end{itemize}
 We  recall that some
classical results on groups and semigroups of linear operators (see
\cite{EngelNagel, Hille, Kantorovitz, Lunardi})
have been extended
to the quaternionic setting in some recent papers.
In  \cite{MR2803786} it has been proved the quaternionic Hille-Yosida theorem,
 in \cite{perturbation} has been studied the problem of generation by perturbations of the quaternionic infinitesimal generator and in
 \cite{FUCGEN}  the natural functional calculus has been defined
 for the infinitesimal generator of quaternionic groups of operators.
For semigroups over real alternative *-algebras and generation theorems
 see \cite{GR}.
\\
\\
Regarding the case of  $(n+1)$-tuples of operators
we postpone the details in Section \ref{sec7}.
We only  mention that the $H^\infty$ functional calculus of $(n+1)$-tuples of operators applies to the Dirac operator.
We point out that the above formulas for the quaternionic functional calculus hold also for $(n+1)$-tuples of noncommuting operators.  Here we just recall the notion of $S$-spectrum.
By $V_n$ we denote the tensor product of the real Banach space $V$ with the real Clifford algebra $\mathbb{R}_n$ we  consider.
In the case of  $(n+1)$-tuples of bounded operators $(T_0,T_1,...,T_n)$ the $S$-spectrum is defined as
$$
\sigma_S(T)=\{s\in \mathbb{H} \ : \   T^2-2s_0T+|s|^2\mathcal{I} \  \ {\rm is \ not\ invertible\ in} \ \mathcal{B}(V_n) \}
$$
where the operators $T_\ell$  act on $V$ for $\ell=0,...,n$.
  The   paravector operator
 $$
 T=T_0+e_1T_1+...+e_nT_n
 $$
   represents the $(n+1)$-tuples of bounded operators $(T_0,T_1,...,T_n)$
 where $e_1,...,e_n$ are the units of the Clifford algebra $\mathbb{R}_n$, $s$ is the paravector $s=s_0+s_1e_1+...+s_ne_n$, with $s_\ell\in \mathbb{R}$ for $\ell=0,...,n$ and $|s|$ is the Euclidean norm of the paravector $s$.
 With these notations in mind also for the $(n+1)$-tuples of bounded operators $(T_0,T_1,...,T_n)$ we can define the $S$-resolvent operators and the $S$-functional calculus, see \cite{CAUCHY, JFA, MR2752913}.
\\
\\
Using a different approach, based on the classical theory of functions in the kernel of the Dirac operator, see \cite{bds,csss}, A. McIntosh with some of his coauthors developed the monogenic functional calculus, see \cite{jefferies,jmc,jmcpw,mcp,LIATAO}, and based on it he also developed the $H^\infty$ functional calculus for commuting operators see \cite{LIMC}.
\\
\\
The plan of the paper is as follows. In Section 2 we recall the main facts on slice hyperholomorphic functions.
In Section 3 we study the rational functions in the quaternionic setting and we define the rational functional calculus.
Section 4 contains the $S$-functional calculus for quaternionic linear operators of type $\omega$ and some properties.
Section 5 is devoted to the definition and some properties of the $H^\infty$ functional calculus for quaternionic operators and in Section 6 we consider quadratic estimates that guarantee the boundedness of the $H^\infty$ functional calculus.
In section 7 we adapt the results of the previous sections to the case of $(n+1)$-tuples of noncommuting operators.

\section{Preliminary results on slice hyperholomorphic functions}\label{SEc2}

In this section we recall some basic facts on the theory of slice hyperholomorphic functions in the quaternionic setting; for the proofs of the statements see the book \cite{MR2752913} and the references therein.
We denote by $\mathbb{S}$ the $2$-sphere of purely imaginary quaternions of modulus 1:
$$
\mathbb{S}=\{q=x_1e_1+x_2e_2+x_3e_3\in \mathbb{H}\ |\ q^2=-1\}
$$
and we recall that for any $i\in \mathbb{S}$ we can define a complex plane $\mathbb{C}_i$ whose elements are of the form $q=u+iv$ for $u$, $v\in \mathbb{R}$.
  Any quaternion $q$ belongs to a suitable  complex plane: if we set
 \[i_q := \begin{cases}\frac{\underline{q}}{|\underline{q}|},& \text{if  }\underline{q} \neq 0 \\ \text{any }i\in\S, \quad&\text{if }\underline{q}  = 0,\end{cases}\]
 then $q = u + i_q v$ with $u =\Re(q)$ and $v = |\underline{q}|$,
 so, it follows that,  the skew field of quaternions $\mathbb{H}$ can be seen as
$$
\mathbb{H}=\bigcup_{i\in \mathbb{S}}\C_i.
$$
\begin{definition}[Slice hyperholomorphic function]\label{defslhyp}
Let $U\subset\H$ be open and let $f: U \to \H$ be a real differentiable function. For any $i\in\S$, let
$
f_i := f|_{U\cap\C_i}
$
denote the restriction of $f$ to the plane $\C_i$. The function $f$ is called left slice hyperholomorphic if, for any $i\in\S$,
\begin{equation}\label{LHolEQ}
\frac12\left( \frac{\partial}{\partial u} f_i(q) + i \frac{\partial}{\partial v} f_i(q)\right) = 0 \quad \text{for all }q = u + iv\in U\cap\C_i\end{equation}
and right slice hyperholomorphic if, for any $i\in\S$,
\begin{equation}\label{RHolEQ}\frac12\left( \frac{\partial}{\partial u} f_i(q) + \frac{\partial}{\partial v} f_i(q)i\right) = 0 \quad \text{for all }q = u + i v\in U\cap\C_i.\end{equation}
A left (or right) slice  hyperholomorphic function that satisfies $f(U\cap\C_i)\subset\C_i$ for every $i\in\S$ is called intrinsic.
\\
We denote the set of all left  slice hyperholomorphic functions on $U$ by $\lhol(U)$, the set of all right slice hyperholomorphic functions on $U$ by $\rhol(U)$ and the set of all intrinsic functions by $\intrin(U)$.
\end{definition}
The importance of the class of intrinsic functions is due to the fact that the multiplication and composition with intrinsic functions preserve slice hyperholomorphicity. This is not true for arbitrary slice hyperholomorphic functions.
\begin{theorem}
Let $U$ be an open set in $\mathbb{H}$ and let $\lhol(U)$, $\rhol(U)$ and $\intrin(U)$ the spaces of slice hyperholomorphic functions defined above.
\begin{itemize}
\item
 If $f\in\intrin(U)$ and $g\in\lhol(U)$, then $fg\in\lhol(U)$.
\item
 If $f\in\rhol(U)$ and $g\in\intrin(U)$, then $fg\in\rhol(U)$.
\item
 If $g\in\intrin(U)$ and $f\in\lhol(g(U))$, then $f\circ g\in \lhol(U)$.
 \item
 If $g\in\intrin(U)$ and $f\in\rhol(g(U))$, then $f\circ g\in \rhol(U)$.
\end{itemize}
\end{theorem}
\begin{remark}\label{rmk23} {\rm As a consequence of the above theorem, we have that
intrinsic functions on $U$ are both left and right slice hyperholomorphic. As we shall see, the set $\intrin(U)$ is a commutative real subalgebra (with respect to a suitable product) of  $\lhol(U)$ and also of $\rhol(U)$.
This fact is of crucial importance for the definition of the $H^\infty$ functional calculus.}
\end{remark}

It is possible to introduce slice hyperholomorphic functions in different ways, see \cite{Global}.
Using Definition \ref{defslhyp},
to prove the most important results of this class of functions, such as the Cauchy formula, we need additional conditions on the open set $U$
that we introduced below.
For any $q=u+i_qv\in\mathbb H$ we define the set
$
 [q] := \{u +iv \ | \ i\in\S\}.
 $
By direct computations it follows that
an element $\tilde q$ belongs to $[q]$ if and only if it is of the form $\tilde q=r^{-1}q r$
for some $r\not=0$ and that $[q]$ is a $2$-sphere.

\begin{definition}
Let $U \subseteq \mathbb{H}$. We say that $U$ is axially symmetric
 if, for all $u+iv \in U$, the whole 2-sphere $[u+iv]$ is contained in $U$.
\end{definition}
\begin{definition}
Let $U \subseteq \mathbb{H}$ be a domain in $\mathbb{H}$. We say that $U$ is a
\textnormal{slice domain (s-domain for short)} if $U \cap \mathbb{R}$ is nonempty and
if $ U\cap\mathbb{C}_i$ is a domain in $\mathbb{C}_i$ for all $i \in \mathbb{S}$.
\end{definition}
We recall that a domain is  an open set that is also simply connected.
\\
To define rational functions and the associated rational functional calculus we are in need of a few more  properties of slice hyperholomorphic functions.

\begin{theorem}[Representation Formula]\label{formulaMON} Let $U$ be an axially symmetric s-domain $U \subseteq  \mathbb{H}$.
\\
Let $f\in \lhol(U)$.  Choose any
$j\in \mathbb{S}$.  Then the following equality holds for all $x=u+iv \in U $:
\begin{equation}\label{distribution}
f(u+iv) =\frac{1}{2}\Big[   f(u+jv)+f(u-jv)\Big] +i\frac{1}{2}\Big[ j[f(u-jv)-f(u+jv)]\Big].
\end{equation}
Moreover, for all $u, v \in \mathbb{R}$ such that $[u+iv] \subseteq U $, the functions
\begin{equation}\label{cappa}
\alpha(u,v)=\frac{1}{2}\Big[   f(u+jv)+f(u-jv)\Big] \ \ {\sl and}
\ \  \beta(u,v)=\frac{1}{2}\Big[ j [f(u-jv)-f(u+jv)]\Big]
\end{equation}
depend on $u,v$ only.
\\
Let $f\in \rhol(U)$.  Choose any
$j\in \mathbb{S}$.  Then the following equality holds for all $x=u+iv \in U $:
\begin{equation}\label{distributionright}
f(u+iv) =\frac{1}{2}\Big[   f(u+jv)+f(u-jv)\Big]+\frac{1}{2}\Big[ [f(u-jv)-f(u+jv)]j\Big]i.
\end{equation}
Moreover, for all $u, v \in \mathbb{R}$ such that $[u+iv] \subseteq U $, the functions
\begin{equation}\label{capparight}
\alpha(u,v)=\frac{1}{2}\Big[   f(u+jv)+f(u-jv)\Big] \ \ {\sl and}
\ \ \beta(u,v)=\frac{1}{2}\Big[ [f(u-jv)-f(u+jv)]j\Big]
\end{equation}
depend on $u,v$ only.
\end{theorem}
\begin{lemma}[Splitting Lemma]\label{Splitting}
Let $U\subseteq \mathbb{H}$ be an open set.
\\
Let  $f\in  \lhol(U)$.  Then for
every $i \in \mathbb{S}$, and every $j\in\mathbb{S}$
perpendicular to $i$, there are two holomorphic functions
$F,G:U\cap \mathbb{C}_i \to \mathbb{C}_i$ such that for any $z=u+iv$, it is
$$
f_i(z)=F(z)+G(z)j.
$$
Let  $f\in  \rhol(U)$.
Then for
every $i \in \mathbb{S}$, and every $j\in\mathbb{S}$,
perpendicular to $i$, there are two holomorphic functions
$F,G:U\cap \mathbb{C}_i \to \mathbb{C}_i$ such that for any $z=u+iv$, it is
$$
f_i(z)=F(z)+jG(z).
$$
\end{lemma}
\begin{remark}
{\rm
In the Splitting Lemma the two holomorphic functions $F$ and $G$ depend on the complex plane $\mathbb{C}_i$ that we consider.
}
\end{remark}

The Cauchy formula for slice hyperholomorphic functions, with a slice hyperholomorphic kernel, is the key tool to define the $S$-functional calculus.
Such formula has two different Cauchy kernels according to right or left slice hyperholomorphicity;
these kernels have power series expansions for $|q|<|s|$:
$\sum_{n=0}^\infty q^ns^{-1-n}$ (in the left case), and $\sum_{n=0}^\infty s^{-1-n}q^n$ (in the right case).
The sum of the first series leads to the definition of the left slice hyperholomorphic Cauchy kernel; analogously the sum of the second series gives
the right slice hyperholomorphic Cauchy kernel.
\begin{definition}
The  left slice hyperholomorphic Cauchy kernel is \[S_L^{-1}(s,q) = -(q^2-2\Re(s)q + |s|^2)^{-1}(q-\overline{s})\quad\text{for }q\notin[s]\]
and the right slice hyperholomorphic Cauchy kernel is
\[S_R^{-1}(s,q) = -(q-\overline{s})(q^2-2\Re(s)q + |s|^2)^{-1}\quad\text{for }q\notin[s].\]
\end{definition}
So we can state the Cauchy formulas:
\begin{theorem}\label{Cauchy}
Let $U\subset\H$ be an axially symmetric slice domain such that its boundary $\partial (U\cap\C_i)$ in $\C_i$ consists of a finite number of continuously differentiable Jordan curves. Let $i\in\S$ and set $ds_i = -i\, ds$. If $f$ is left slice hyperholomorphic on an open set that contains $\overline{U}$, then
\[f(q) = \frac{1}{2\pi}\int_{\partial(U\cap\C_i)} S_L^{-1}(s,q)\,ds_i\, f(s)\quad\text{for all }q\in U.\]
If $f$ is right slice hyperholomorphic on an open set that contains $\overline{U}$, then
\[f(q) = \frac{1}{2\pi}\int_{\partial(U\cap\C_i)}f(s)\, ds_i\, S_R^{-1}(s,q)\quad\text{for all }q\in U.\]
The above integrals do not depend neither on the open set $U$ nor on the complex plane $\mathbb{C}_i$ for $i\in \mathbb{S}$.
\end{theorem}

\begin{theorem}
The left slice hyperholomorphic Cauchy kernel $S_L^{-1}(s,q)$ is left slice hyperholomorphic in the variable $q$ and right slice hy\-per\-ho\-lo\-mor\-phic in the variable $s$ in its domain of definition (a similar result holds for $S_R^{-1}(s,q)$). Moreover, we have $S_R^{-1}(s,q) = - S_L^{-1}(q,s)$.
\end{theorem}
In the sequel
we will be in need of the following theorem:
\begin{theorem}[Cauchy's integral theorem]
Let $U\subset\H$ be an open set, let $i\in\S$ and let $D_i$ be an open subset of $U\cap\C_i$ with $\overline{D_i}\subset U\cap\C_i$ such that its boundary $\partial D_i$ consists of a finite number of continuously differentiable Jordan curves. For any $f\in\rhol(U)$ and $g\in\lhol(U)$, it is
\[\int_{\partial D_i}f(s)\,ds_i\,g(s) = 0,\]
where $ds_i = -i\ ds$.
\end{theorem}

\section{Rational functions and their functional calculus}

Let $V$ be a two-sided quaternionic Banach space. We denote the set of all bounded quaternionic right-linear operators on $V$ by $\boundOP(V)$.
In the quaternionic setting, in particular for unbounded operators, we have to specify if
 we are considering  a left-linear or a right-linear operator.
When some properties of a quaternionic operator depend just on linearity we simply say (quaternionic) linear operator and
 we do not specify the type of linearity.
In analogy with the complex case, we say that a linear
operator, whose domain  $\mathcal{D}(T):=\{v\in V\ :\ Tv\in V\}$, is closed if its graph is closed.
\begin{definition}
We define the $S$-resolvent set of a linear closed operator $T$ as
\[\rho_S(T):= \{ s\in\H \ :\  (T^2 - 2\Re(s)T + |s|^2\id)^{-1}\in\boundOP(V)\},\]
where
$$
T^2-2 {\rm Re}(s) T +|s|^2\mathcal{I}\, :\, \mathcal{D}(T^2)\to V,
$$
and the $S$-spectrum of $T$ as
\[\sigma_S(T):=\H\setminus\rho_S(T).\]
\end{definition}
If we consider bounded linear  operators the $S$-spectrum is a compact and nonempty set in $\mathbb{H}$, but in the case of unbounded operators the $S$-spectrum can be every closed subset of $\mathbb{H}$, also an empty set.
In the sequel, when we consider unbounded operators we will assume that the $S$-resolvent set is nonempty.

For $n=0,1,2,...$, the powers of an operator $T$ are defined inductively by the relations $T^0=\mathcal{I}$, $T^1=T$ and
$$
\mathcal{D}(T^n)=\{v\ :\ v\in \mathcal{D}(T^{n-1}),\ \ T^{n-1}v\in \mathcal{D}(T)\ \},
$$
$$
T^nv=T(T^{n-1}v),\ \ \ v\in \mathcal{D}(T^n).
$$
For  $a_\ell\in \mathbb{H}$,  $\ell=0,\ldots ,m$,  the polynomial
$$
P_m(q)=\sum_{\ell=0}^m a_\ell q^\ell,
$$
of degree $m\in \mathbb{N}$, is right slice hyperholomorphic.
The natural functional calculus for polynomials is obtained by replacing $q$ by the right (resp. left) linear operator $T$. We obtain
the right (resp. left)  linear quaternionic operator
$$
P_m(T)=\sum_{\ell=0}^m a_\ell T^\ell: \mathcal{D}(T^m)\to V.
$$
Analogous considerations can be done when we consider a left hyperholomorphic  polynomial
$P_m(q)=\sum_{\ell=0}^m q^\ell a_\ell$ of degree $m\in \mathbb{N}$ and
the  right (resp. left) linear quaternionic operator
$$
P_m(T)=\sum_{\ell=0}^m T^\ell a_\ell: \mathcal{D}(T^m)\to V
$$
is obtained  replacing $q$ by a right (resp. left) linear quaternionic operator $T$.
\\
\\
An important ingredient in the definition of the $H^\infty$ functional calculus
is the rational functional calculus.
To define the quaternionic $H^\infty$ functional calculus  we have to define the rational functional calculus for intrinsic functions.
 As we have already observed in Remark \ref{rmk23}, this class consists of functions that are both left and right slice hyperholomorphic. Thus we first give the definition of rational left  slice hyperholomorphic
functions and then we consider the subset of rational intrinsic functions.
\\
\\
Suppose that $U \subseteq \mathbb{H}$ is an axially symmetric $s$-domain and let $f$ and $g:U\to \mathbb{H}$ be  left slice hyperholomorphic functions. For any $i,j\in\mathbb{S}$, with  $i\perp j$, the Splitting Lemma \ref{Splitting} guarantees the existence of four holomorphic functions
 $F,G,H,K: \ U\cap\mathbb{C}_i\to \mathbb{C}_i$ such
 that for all $z=x+iy\in   U\cap\mathbb{C}_i$
 \begin{equation}\label{fandg}
 f_i(z)=F(z)+G(z)j \qquad g_i(z)=H(z)+K(z)j.
 \end{equation}
We define the function $f_i\star g_i:\ U\cap\mathbb{C}_i\to \mathbb{H}$ as
\begin{equation}\label{f*g}
f_i\star g_i(z)=[F(z)H(z)-G(z)\overline{K(\bar z)}]+[F(z)K(z)+G(z)\overline{H(\bar z)}]j.
\end{equation}
We now note that the Representation Formula provides an extension operator denoted by  ${\rm ext}$, see \cite{MR2752913}, and so can give the following definition:
 \begin{definition}
Let $U\subseteq\mathbb{H}$ be an axially symmetric s-domain and let
 $f,g:\ U\to\mathbb{H}$  be left slice hyperholomorphic. The function
$$
(f\star g)(q)={\rm ext}(f_i\star g_i)(q)$$ defined as the extension of (\ref{f*g})
is called the slice hyperholomorphic  product of $f$
and $g$ and is denoted by $f\star g$.
 \end{definition}
 \begin{remark}{\rm Let $U$ be an axially symmetric s-domain. The sets $\lhol(U)$ and $\rhol(U)$ equipped with the operation of sum and $*$-product turn out to be non commutative, unital, real algebras. $\mathcal{N}(U)$ is a real subalgebra of both of them.}
 \end{remark}
 \begin{Ex}{\rm
In the case $f$ and $g$ have power series expansion:
$\sum_{n=0}^\infty q^n a_n$ and $\sum_{n= 0}^\infty q^n b_n$, respectively for $a_n,b_n\in\mathbb{H}$, then the
slice hyperholomorphic product becomes
\begin{equation}\label{*perserie}
\Big(\sum_{n=0}^\infty q^n a_n\Big) \star \Big(\sum_{n= 0}^\infty q^n b_n\Big)=\sum_{n= 0}^\infty q^n \sum_{k=0}^na_{n-k}b_k.
\end{equation}
}
\end{Ex}
\begin{remark}
{\rm
For more comments on the slice hyperholomorphic  product and all its consequences see the book \cite{MR2752913}.
}
\end{remark}
Let $f$ be as above, and let its restriction to $\mathbb C_i$ be as in \eqref{fandg}.
We define the function $f^s_i:\ U\cap\mathbb{C}_i\to \mathbb{C}_i$ as
\begin{equation}\label{f^s_qua}
f^s_i(z):=F(z)\overline{F(\bar z)}+G(z)\overline{G(\bar z)},
\end{equation}
and we set
$$
f^s(q)={\rm ext}(f^s_i)(q).
$$
\begin{definition}[Slice inverse function]
Let $U\subseteq \mathbb{H}$  be an axially symmetric s-domain and let $f : U \to \mathbb{H}$
be a left slice hyperholomorphic function.
We define the function $f^{-\star}$ as
$$
f^{-\star}(q):=(f^s(q))^{-1}f^c(q),
$$
where $f^c(q)$ is the slice hyperholomorphic extension of
$$
f_i^c(z)=\overline{F(\bar z)}-G(z)j.
$$
The function $f^{-\star}$ is defined on $U$ outside the zeros of $f^s$.
\end{definition}
Let $Q(s)$ be a polynomial (with quaternionic coefficient on the right); then   $Q^{-\star}(s)$ is a rational function. Given a polynomial $P(s)$, the quaternionic rational functions are of the form
$$
R(s)=(P\star Q^{-\star})(s),
$$
or
$$
R(s)= (Q^{-\star}\star P)(s),
$$
so, in principle, one should make a choice between right quotient (first case) or left quotients (second case). In case $Q(s)$ has real coefficients, the two choices are equivalent, and this will be the case we will consider in this paper.
\begin{definition}[Rational functional calculus]
Let $R=P*Q^{-*}$ be a rational function  and assume
that $R$ has no poles on the S-spectrum of $T$.  Let $T$ be a closed densely defined operator.
 We define the rational functional calculus as:
$$
R(T)=(P\star Q^{-*})(T)
$$
(or  $R(T)=(Q^{-*}\star P )(T)$).
\end{definition}
The operator $R(T)$ is closed and densely defined and its domain is $\mathcal{D}(T^m)$ where
$$
m:=\max\{0, \deg P-\deg Q\}.
$$
\begin{remark}
For our purposes, we will consider intrinsic rational functions defined on $U$. In this case
$$
R\star R_1=R R_1=R_1 R=R_1\star R,
$$
for every $R$ and $R_1$ intrinsic rational functions.
Moreover, if $P(s)$ and $Q(s)$ are intrinsic polynomial then  $(P\star Q^{-*})(s)=(PQ^{-1})(s)=(Q^{-1}P)(s)$.
Intrinsic rational functions constitute a real, commutative, subalgebra of both the real algebras $\lhol(U)$ and   $\rhol(U)$.
\end{remark}
\begin{Ex}{\rm
An important example of intrinsic rational function, useful in the sequel, is
$$
\psi_k(s)=\Big(\frac{s}{1+s^2}\Big)^k, \ \ k\in \mathbb{N}.
$$
Note that, for the sake of simplicity, from now on we will write $\psi(s)$ instead of $\psi_k(s)$.
\\
We recall that slice hyperholomorphic rational functions have poles that are real points and/or spheres. This is compatible with the  structure of the $S$-spectrum of $T$ that consists of real point and/or spheres, see p.142 in \cite{MR2752913}.\\
With $\psi$ as above, we have
$$
\psi(T)= \Big(T(\mathcal{I}+T^2)^{-1}\Big)^k, \ \ k\in \mathbb{N}.
$$
}
\end{Ex}
We summarize in the following the properties of the rational functional calculus. 
The proofs are similar to the classical results and for this reason we omit them.
\begin{Pn}\label{RATUNO} Let $T$ be a linear quaternionic operator single valued on a quaternionic Banach space $V$.
Let $P$ and $Q$ be  intrinsic quaternionic polynomials of order $n$ and $m$, respectively.
 Then
\begin{itemize}
\item[(i)]
If $P\not\equiv 0$ then $P(T)Q(T)=(PQ)(T)$.
\item[(ii)]
If $P(T)$ is injective and $Q\not\equiv 0$ then
$$
\mathcal{D}(P(T)^{-1})\cap \mathcal{D}(Q(T))\subset \mathcal{D}(P(T)^{-1}Q(T))\cap \mathcal{D}(Q(T)P(T)^{-1})
$$
and
$$
P(T)^{-1}Q(T)v=Q(T)P(T)^{-1}v,\ \ \ \ \forall v\in \mathcal{D}(Q(T))\cap \mathcal{D}(P(T)^{-1}).
$$
\item[(iii)]
 Suppose that $T$ is a closed linear operator with $\rho_S(T)\not=\emptyset$. Then
$P(T)$ is closed and $P(\sigma_S(T))=\sigma_S(P(T))$.
\end{itemize}
\end{Pn}
For rational functions we have
\begin{Pn}\label{propratio}
Let $T$ be a linear quaternionic operator single valued on a quaternionic Banach space $V$ with $\rho_S(T)\not=\emptyset$.
Let $0\not\equiv R=PQ^{-1}$ and $R_1=P_1Q_1^{-1 }$ be intrinsic rational functions. Then we have:
\begin{itemize}
\item[(i)]
$R(T)$ is a closed operator.
\item[(ii)]
$R(\overline{\sigma_S}(T))\subset \overline{\sigma_S}(R(T))$,
where $\overline{\sigma_S}(T)=\sigma_S(T)\cup\{\infty\}$ denotes the extended $S$-spectrum of $T$.
\item[(iii)]
$R(T)R_1(T)\subset (RR_1)(T)$ and  equality holds if
$$
(\deg(P)-\deg(Q))(\deg(P_1)-\deg(Q_1))\geq 0.
$$
\item[(iv)]
$R(T)+R_1(T)\subset (R+R_1)(T)$ and  equality holds if
$$
\deg(PQ_1+P_1Q)=\max\{\deg(PQ_1), \deg(P_1Q)\}.
$$
\end{itemize}
\end{Pn}

\section{The $S$-functional calculus for quaternionic operators of type $\omega$}
As we have mentioned in the introduction we want to show that, at least for a suitable subclass of closed densely defined operators, we can extend the
formulas of the $S$-functional calculus for bounded operators. In order to do this we need to define the $S$-resolvent operators as follows.
Let $T$ be a closed linear operator on a two-sided quaternionic Banach space $V$ and assume that $s\in\rho_S(T)\not=\emptyset$, then the operator
\[
Q_s(T):=(T^2-2\Re(s)T +|s|^2\id )^{-1}
\]
is called the pseudo-resolvent of $T$.
\\
 We will denote the set of all closed  quaternionic right-linear operators on the two-sided quaternionic Banach space $V$ by $\closOP(V)$; in the case of left-linear operators we will use the notation $\closOP_L(V)$.
\begin{definition}
Let $T\in\closOP(V)$. The left $S$-resolvent operator is defined as
\begin{equation}\label{SresolvoperatorL}
S_L^{-1}(s,T):= Q_s(T)\overline{s} -TQ_s(T),\ \ \ \ s\in\rho_S(T),
\end{equation}
and the right $S$-resolvent operator is defined as
\begin{equation}\label{SresolvoperatorR}
S_R^{-1}(s,T):=-(T-\id \overline{s})Q_s(T),\ \ \ \ s\in\rho_S(T).
\end{equation}
\end{definition}
In the sequel we will work just with right linear operators and the above $S$-resolvent operators are slice hyperholomorphic functions operator-valued
and $S_L^{-1}(s,T)v$ and $S_R^{-1}(s,T)v$ are defined for all $v\in V$.
\begin{remark}{\em
We point out that  in the case  $T\in\closOP_L(V)$ the
$S$-resolvent operators have to be defined as
\begin{equation}
S_L^{-1}(s,T):= Q_s(T)(\overline{s} -T),\ \ \ \ s\in\rho_S(T),
\end{equation}
and the right $S$-resolvent operator is defined as
\begin{equation}\label{SresolvoperatorR}
S_R^{-1}(s,T):= \overline{s}Q_s(T)-TQ_s(T),\ \ \ \ s\in\rho_S(T).
\end{equation}
}
\end{remark}

The $S$-resolvent equation has been proved in \cite{acgs} when the operator $T$ is bounded.
 For the case of unbounded operators one has to be more careful with the domains of the operators  that are involved, see \cite{CG}.
The resolvent equation of the $S$-functional calculus and the following Lemma \ref{Lemma321}
 will be important in the proofs of some properties
of the $S$-functional calculus for operators of type $\omega$.
\begin{theorem}[$S$-resolvent equation, see \cite{CG}]
Let $T\in\closOP(V)$. For  $s,p \in  \rho_S(T)$ with $s\notin[p]$, it is
\begin{equation}\label{resEQ}
\begin{split}
S_R^{-1}(s,T)S_L^{-1}(p,T)v=&\big[[S_R^{-1}(s,T)-S_L^{-1}(p,T)]p
\\
&
-\overline{s}[S_R^{-1}(s,T)-S_L^{-1}(p,T)]\big](p^2-2s_0p+|s|^2)^{-1}v, \ \ \ v\in V.
\end{split}
\end{equation}
\end{theorem}
We recall the following lemma from \cite{acgs}.
\begin{lemma}\label{Lemma321} Let $B\in \mathcal{B}(V)$. Let $U$ be an axially symmetric s-domain and assume that
$f\in \mathcal{N}(U)$.
Then, for $p\in U$, we have
$$
\frac{1}{2\pi}\int_{\partial(U\cap\mathbb{C}_i)}f(s)ds_i
(\overline{s}B-Bp)(p^2-2s_0p+|s|^2)^{-1}=Bf(p),
$$
where $ds_i = -i\, ds$.
\end{lemma}
We can now consider the main definitions for the $H^\infty$ functional calculus for quaternionic operators.

\begin{definition}[Argument function]
Let $s\in\hh\setminus\{0\}$. We define $\arg(s)$ as the unique number $\theta \in [0,\pi]$ such that $s = |s| e^{\theta i_s}$.
\end{definition}
Observe that $\theta = \arg(s)$ does not depend on the choice of $i_s$ if $s\in\rr\setminus\{0\}$ since $p = |p|e^{0 i}$ for any $i\in\S$ if $p>0$ and $p = |p|e^{\pi i}$ for any $i\in\S$ if $p<0$.

Let $\vartheta \in [0, \pi]$ we  define the sets
$$
\mathcal{S}_\vartheta=\{ s\in \mathbb{H}\ |\ \ |\arg(p)|\leq \vartheta \ {\rm or}\  s=0\},
$$
\begin{equation}\label{esse0}
\mathcal{S}^0_\vartheta=\{ s\in \mathbb{H}\ |\ \ |\arg(p)|< \vartheta\}.
\end{equation}

\begin{definition}[Operator of type $\omega$]
Let $\omega\in [0,\pi)$ we say the the linear operator  $T:D(T)\subseteq V\to V$  is of type $\omega$ if
\begin{itemize}
\item[(i)]
$T$ is closed and densely defined
\item[(ii)]
$\sigma_S(T)\subset \mathcal{S}_\vartheta\cup\{\infty\}$
\item[(iii)]
for every $\vartheta\in (\omega,\pi]$ there exists a positive constant $C_\vartheta$ such that
$$
\|S_L^{-1}(s,T)\|\leq \frac{C_\vartheta}{|s|}\ \ {\rm for  \ all  \ non\ zero}\  s\in\mathcal{S}^0_\vartheta,
$$
$$
\|S_R^{-1}(s,T)\|\leq \frac{C_\vartheta}{|s|}\ \ {\rm for  \ all  \ non\ zero}\  s\in\mathcal{S}^0_\vartheta.
$$
\end{itemize}
\end{definition}
We now introduce the following subsets of the set of slice hyperholomorphic functions that consist of bounded slice hyperholomorphic functions.
\begin{definition}
Let $\mu\in (0,\pi]$. We set
$$
\mathcal{SH}_L^\infty(\mathcal{S}^0_\mu)=\{f\in \mathcal{SH}_L(\mathcal{S}^0_\mu)\ \ \ {\rm such \ that}\ \|f\|_\infty:=\sup_{s\in\mathcal{S}^0_\mu}|f(s)|<\infty \},
$$
$$
\mathcal{SH}_R^\infty(\mathcal{S}^0_\mu)=\{f\in \mathcal{SH}_R(\mathcal{S}^0_\mu)\ \ \ {\rm such \ that}\ \|f\|_\infty:=\sup_{s\in\mathcal{S}^0_\mu}|f(s)|<\infty \},
$$
$$
\mathcal{N}^\infty(\mathcal{S}^0_\mu):=\{f\in \mathcal{N}(\mathcal{S}^0_\mu)\ \ \ {\rm such \ that}\ \|f\|_\infty:=\sup_{s\in\mathcal{S}^0_\mu}|f(s)|<\infty\}.
$$
\end{definition}
In order to define bounded functions of operators of type $\omega$, we need to introduce suitable
subclasses of bounded slice hyperholomorphic functions:
\begin{definition} We define the spaces
$$
\Psi_L(\mathcal{S}^0_\mu)=\{f\in \mathcal{SH}_L^\infty(\mathcal{S}^0_\mu)\ {\it such \ that} \ \exists \ \alpha>0,\ c>0\
|f(s)|\leq \frac{c|s|^\alpha}{1+|s|^{2\alpha}},\ \ {\it for \ all}\ s\in \mathcal{S}^0_\mu\},
$$
$$
\Psi_R(\mathcal{S}^0_\mu)=\{f\in \mathcal{SH}_R^\infty(\mathcal{S}^0_\mu)\ {\it such \ that} \ \exists \ \alpha>0,\ c>0\
|f(s)|\leq \frac{c|s|^\alpha}{1+|s|^{2\alpha}},\ \ {\it for \ all}\ s\in \mathcal{S}^0_\mu\},
$$
$$
\Psi(\mathcal{S}^0_\mu)=\{f\in \mathcal{N}^\infty(\mathcal{S}^0_\mu)\ {\it such \ that} \ \exists \ \alpha>0,\ c>0\
|f(s)|\leq \frac{c|s|^\alpha}{1+|s|^{2\alpha}},\ \ {\it for \ all}\ s\in \mathcal{S}^0_\mu\}.
$$
\end{definition}
The following theorem is a crucial step for the definition of the $S$-functional calculus, because it shows that the following integrals  depend neither on the path that we choose nor on the complex plane $\mathbb{C}_i$, $i\in \mathbb{S}$.
\begin{theorem}\label{theoindep}
Let $T$ be an operator of type $\omega$.
Let $i\in\S$, and let $\mathcal{S}^0_\mu$ be as in \eqref{esse0}. Choose a piecewise smooth path $\Gamma$ in $\mathcal{S}^0_\mu\cap\C_i$ that goes from $\infty e^{i\theta}$ to $\infty e^{-i\theta}$, where $\omega <\theta <\mu$. Then the integrals
\begin{equation}\label{PsiTL}
\frac{1}{2\pi}\int_{\Gamma}S_L^{-1}(s,T) \, ds_i\, \psi(s),\ \ \ \ {\it for\ all}\ \ \ \psi\in \Psi_L(\mathcal{S}^0_\mu),
\end{equation}
\begin{equation}\label{PsiTR}
 \frac{1}{2\pi}\int_{\Gamma} \psi(s)\, ds_i\, S_R^{-1}(s,T),\ \ \ \ {\it for\ all}\ \ \ \psi\in \Psi_R(\mathcal{S}^0_\mu),
\end{equation}
 depend neither on $\Gamma$ nor on $i\in \mathbb{S}$, and they define bounded operators.
\end{theorem}
\begin{proof}
We reason on the integral (\ref{PsiTL}) since (\ref{PsiTR}) can be treated in a similar way.

The growth estimates on $\psi$ and on the resolvent operator imply that the integral (\ref{PsiTL}) exists
and defines a bounded right-linear operator.

 The independence of the choice of $\theta$ and of the choice of the path $\Gamma$ in the complex plane $\C_i$ follows from Cauchy's integral theorem.

In order to show that the integral  (\ref{PsiTL}) is independent of the choice of the imaginary unit $i\in\S$,
we take an arbitrary $j\in\S$ with $j\neq i$.
\\
Let $B(0,r)$ the ball centered at the origin with radius $r$;
let $a_0 > 0 $ and $\theta_0\in(0,\pi)$, $n\in\N$, we define the sector
$\Sigma(\theta_0, a_0)$ as
\[\Sigma(\theta_0, a_0) := \{s\in\hh: \arg(s-a_n)\geq \theta_n\}.
\]

 Let $\theta_0 < \theta_s < \theta_p <\pi$ and set $U_s := \Sigma(\theta_s, 0)\cup B(0,a_0/2)$
 and
 $U_p:= \Sigma(\theta_p, 0)\cup B(0,a_0/3)$,

  where the indices $s$ and $p$ denote the variable of integration over the boundary of the respective set.

    Suppose that $U_p$ and $U_s$ are  slice domains and $\partial(U_s\cap\C_i)$ and $\partial(U_p\cap\C_j)$ are paths that are contained in the sector.

     Observe that $\psi(s)$ is right slice hyperholomorphic on $\overline{U_p}$, and hence, by Theorem~\ref{Cauchy}, we have
\begin{align}
\psi(T) &= \frac{1}{2\pi}\int_{\partial(U_s\cap\C_I)}\psi(s) \, ds_i\, S_R^{-1}(s,T)
 \\
&
= \frac{1}{(2\pi)^2}\int_{\partial(U_s\cap\C_i)}\left(\int_{\partial(U_p\cap\C_j)}\psi(p)\, dp_j\, S_R^{-1}(p,s)\right)\, ds_i\, S_R^{-1}(s,T)
\\
&
= \frac{1}{2\pi}\int_{\partial(U_p\cap\C_i)}\psi(p)\, dp_j \left(\frac{1}{2\pi}\int_{\partial(U_s\cap\C_i)}\, S_R^{-1}(p,s)\, ds_i\, S_R^{-1}(s,T)\right)
\\
&
=  \frac{1}{2\pi}\int_{\partial(U_p\cap\C_j)}\psi(p)\, dp_j\, S_R^{-1}(p,T).
\end{align}
To  exchange order of integration we apply the Fubini theorem.
 The last equation follows as an application of the $S$-functional calculus for unbounded operators,
 see  \cite[Theorem 4.16.7]{MR2752913}, since $S_R^{-1}(p,\infty) = \lim_{s\to\infty} S_R^{-1}(p,s) = 0$. So we get the statement.

\end{proof}

Thanks to the above theorem the following definitions are well posed.

\begin{definition}[The $S$-functional calculus for operators of type $\omega$]\label{psiT}
Let $T$ be an operator of type $\omega$.
Let $i\in\S$, and let $\mathcal{S}^0_\mu$ be the sector defined above. Choose a piecewise smooth path $\Gamma$ in $\mathcal{S}^0_\mu\cap\C_i$ that goes from $\infty e^{i\theta}$ to $\infty e^{-i\theta}$, for $\omega <\theta <\mu$,  then
\begin{equation}\label{PsiTLmon}
\psi(T):= \frac{1}{2\pi}\int_{\Gamma}S_L^{-1}(s,T) \, ds_i\, \psi(s),\ \ \ \ {\it for\ all }\ \ \ \psi\in \Psi_L(\mathcal{S}^0_\mu),
\end{equation}\begin{equation}\label{PsiTRmon}
\psi(T):= \frac{1}{2\pi}\int_{\Gamma} \psi(s)\, ds_i\, S_R^{-1}(s,T),\ \ \ \ {\it for\ all }\ \ \ \psi\in \Psi_R(\mathcal{S}^0_\mu).
\end{equation}
\end{definition}
\begin{remark}{\rm
For functions  $\psi$ that belong to $\Psi(\mathcal{S}^0_\mu)$ both  representations can be used, moreover
$$
\psi(T):= \frac{1}{2\pi}\int_{\Gamma} \psi(s)\, ds_i\, S_R^{-1}(s,T)=\frac{1}{2\pi}\int_{\Gamma}S_L^{-1}(s,T) \, ds_i\, \psi(s),
\ \ \ \ {\it for\ all }\ \ \ \psi\in \Psi(\mathcal{S}^0_\mu).
$$
}
\end{remark}
If $T$ is an operator of type $\omega$, then
 $\psi(T)$,  defined in (\ref{PsiTLmon}) and (\ref{PsiTRmon}), satisfy:
 $$(a\psi+b\varphi)(T)=a\psi(T)+b\varphi(T),\ \ \ \ {\it for \ all}\ \ \ \psi,\varphi\in \Psi_L(\mathcal{S}^0_\mu),$$
$$(a\psi+b\varphi)(T)=a\psi(T)+b\varphi(T),\ \ \ \ {\it for \ all}\ \ \ \psi,\varphi\in \Psi_R(\mathcal{S}^0_\mu).$$
These equalities can be verified with standard computations.
\begin{theorem}\label{PRDRULE} Let $T$ be an operator of type $\omega$. Then
$$ (\psi\varphi)(T)=\psi(T)\varphi(T),\ \ \ \ {\it for \ all}\ \ \ \psi\in \Psi(\mathcal{S}^0_\mu),\ \ \ \varphi\in \Psi_L(\mathcal{S}^0_\mu),
$$
$$ (\psi\varphi)(T)=\psi(T)\varphi(T),\ \ \ \ {\it for \ all}\ \ \ \psi\in \Psi_R(\mathcal{S}^0_\mu),\ \ \ \varphi\in \Psi(\mathcal{S}^0_\mu).$$
\end{theorem}
\begin{proof} We prove the first relation the second one follows with similar computations.
Let $\sigma_{S}(T) \subset U_1$ and $U_2$ be two open sectors that contain the $S$-spectrum  of $T$ and such that
$U_1 \cup \partial U_1 \subset U_2 $ and $U_2 \cup \partial U_2\subset\mathcal{S}^0_\mu$.
Take $p\in \partial (U_1\cap \mathbb{C}_i)$ and $s\in \partial (U_2\cap \mathbb{C}_i)$ and observe that, for $i\in \mathbb{S}$,  the $S$-resolvent equation \eqref{resEQ} implies
\[
\begin{split}
 \psi(T)\varphi(T)&=
 {{1}\over{(2\pi)^2 }} \int_{\partial (U_2\cap \mathbb{C}_i)} \  \psi(s)\ ds_i \ S_R^{-1} (s,T)
  \int_{\partial (U_1\cap \mathbb{C}_i)} \ S_L^{-1} (p,T)\ dp_i \  \varphi(p)
 \\
&=\frac{1}{(2\pi)^2 }\int_{ \partial ( U_2 \cap \mathbb{C}_i) } \psi(s)\ ds_i \int_{ \partial ( U_1 \cap \mathbb{C}_i) }
S_R^{-1}(s,T)p(p^2-2s_0p+|s|^2)^{-1}dp_i\  \varphi(p)
 \\
&-\frac{1}{(2\pi)^2 }\int_{ \partial ( U_2 \cap \mathbb{C}_i) } \psi(s)\ ds_i \int_{ \partial ( U_1 \cap \mathbb{C}_i) }
S_L^{-1}(p,T)p(p^2-2s_0p+|s|^2)^{-1}dp_i\  \varphi(p)
\\
&
-\frac{1}{(2\pi)^2 }\int_{ \partial ( U_2 \cap \mathbb{C}_i) } \psi(s)\ ds_i
\int_{ \partial ( U_1 \cap \mathbb{C}_i) }\overline{s}S_R^{-1}(s,T)(p^2-2s_0p+|s|^2)^{-1}
 dp_i\ \varphi(p)
 \\
&
+\frac{1}{(2\pi)^2 }\int_{ \partial ( U_2 \cap \mathbb{C}_i) } \psi(s)\ ds_i
\int_{ \partial ( U_1 \cap \mathbb{C}_i) }\overline{s}S_L^{-1}(p,T)(p^2-2s_0p+|s|^2)^{-1}
 dp_i\  \varphi(p).
 \end{split}
\]
But since the functions
$$
p\mapsto p(p^2-2s_0p+|s|^2)^{-1}\varphi(p)$$
and
$$
p\mapsto (p^2-2s_0p+|s|^2)^{-1}\varphi(p)
$$
are holomorphic on an open set that contains $\overline{G_{1}\cap\C_i}$, Cauchy's integral theorem implies
$$
\frac{1}{(2\pi)^2 }\int_{ \partial ( U_2 \cap \mathbb{C}_i) } \psi(s)\ ds_i \int_{ \partial ( U_1 \cap \mathbb{C}_i) }
S_R^{-1}(s,T)p(p^2-2s_0p+|s|^2)^{-1}dp_i\  \varphi(p)=0
$$
and
$$
-\frac{1}{(2\pi)^2 }\int_{ \partial ( U_2 \cap \mathbb{C}_i) } \psi(s)\ ds_i
\int_{ \partial ( U_1 \cap \mathbb{C}_i) }\overline{s}S_R^{-1}(s,T)(p^2-2s_0p+|s|^2)^{-1}
 dp_i\  \varphi(p)=0.
$$
It follows that
\[
\begin{split}
\psi(T)\varphi(T)&=
-\frac{1}{(2\pi)^2 }\int_{ \partial ( U_2 \cap \mathbb{C}_i) } \psi(s)\ ds_i \int_{ \partial ( U_1 \cap \mathbb{C}_i) }
S_L^{-1}(p,T)p(p^2-2s_0p+|s|^2)^{-1}dp_i\  \varphi(p)
\\
&
+\frac{1}{(2\pi)^2 }\int_{ \partial ( U_2 \cap \mathbb{C}_i) } \psi(s)\ ds_i
\int_{ \partial ( U_1 \cap \mathbb{C}_i) }\overline{s}S_L^{-1}(p,T)(p^2-2s_0p+|s|^2)^{-1}
 dp_i\  \varphi(p),
 \end{split}
\]
which can be written as
\[
\begin{split}
 \psi(T)\varphi(T)=
 \frac{1}{(2\pi)^2 }\int_{ \partial ( U_2 \cap \mathbb{C}_i) } \psi(s)\ ds_i &\int_{ \partial ( U_1 \cap \mathbb{C}_i) }
[\overline{s}S_L^{-1}(p,T)-S_L^{-1}(p,T)p]\times
\\
&
\times(p^2-2s_0p+|s|^2)^{-1}dp_i\  \varphi(p).
 \end{split}
\]
Using Lemma \ref{Lemma321} we get
\begin{align*}
\psi(T)\varphi(T)&=\frac{1}{2\pi } \int_{ \partial ( U_1 \cap \mathbb{C}_i) }S_L^{-1}(p,T)dp_i \ \psi(p)\  \varphi(p)\\
 &=\frac{1}{2\pi } \int_{ \partial ( U_1 \cap \mathbb{C}_i) }S_L^{-1}(p,T)dp_i (\psi\varphi)(p)\ =(\psi\varphi)(T)
 \end{align*}
 which gives the statement.

\end{proof}

\begin{remark}
We point out that the spectral mapping theorem holds for this functional calculus. Precisely:
if $\psi\in \Psi(\mathcal{S}^0_\mu)$, then $\psi(\sigma_S(T))=\sigma_S(\psi(T))$.
If is important to observe that the spectral mapping theorem holds just for intrinsic functions.
The analogue convergence theorem at p. 216 in the paper of A. McIntosh \cite{McI1} holds also here, in fact it is based on the principle of uniform boundedness which holds also for quaternionic operators. We will not give the proof of this theorem because it is very similar to that classical case.
\end{remark}

\section{The $H^\infty$ functional calculus based on the $S$-spectrum}

To define the $H^\infty$ functional calculus we suppose that $T$ is an operator of type $ \omega$ and moreover we assume that it is one-to-one and with dense range. Here we will consider slice hyperholomorphic functions defined on the open sector $\mathcal{S}^0_\mu$, for  $0\leq \omega <\mu\leq \pi$ which can grow at infinity as $|s|^k$ and at the origin as $|s|^{-k}$ for $k\in \mathbb{N}$. This enlarges the class of functions to which the functional calculus can be applied. Precisely we define:

\begin{definition}[The set $\Omega$]
Let $\omega$ be a real number such that $0\leq \omega \leq \pi$. We denote by $\Omega$ the set of linear operators $T$ acting on a two sided quaternionic Banach space such that:
\begin{itemize}
\item[(i)] $T$ is a linear operator of type $\omega$;
\item[(ii)] $T$ is one-to-one and with dense range.
\end{itemize}
\end{definition}
Then we define the following function spaces according to the set of operators defined above:
\begin{definition}
Let $\omega$ and $\mu$ be  real numbers such that $0\leq \omega < \mu\leq \pi$, we set
$$
\mathcal{F}_L(\mathcal{S}^0_\mu)=\{f\in \mathcal{SH}_L(\mathcal{S}^0_\mu)\ \ \ {\it such \ that}\
|f(s)|\leq C(|s|^k+|s|^{-k})\ \ {\it for\ some}\ \ k >0\  {\it and}\  C>0 \},
$$
$$
\mathcal{F}_R(\mathcal{S}^0_\mu)=\{f\in \mathcal{SH}_R(\mathcal{S}^0_\mu)\ \ \ {\it such \ that}\
|f(s)|\leq C(|s|^k+|s|^{-k})\ \ {\it for\ some}\ \ k>0\  {\it and}\  C>0 \}.
$$
$$
\mathcal{F}(\mathcal{S}^0_\mu)=\{f\in \mathcal{N}(\mathcal{S}^0_\mu)\ \ \ {\it such \ that}\
|f(s)|\leq C(|s|^k+|s|^{-k})\ \ {\it for\ some}\ \ k>0\  {\it and}\  C>0 \}.
$$
\end{definition} To extend the functional calculus
we consider a quaternionic two sided Banach space $V$, the operators in the class $\Omega$, and:
\begin{itemize}
  \item
  The non commutative algebra $\mathcal{F}_L(\mathcal{S}^0_\mu)$ (resp. $\mathcal{F}_R(\mathcal{S}^0_\mu)$).
  \item
  The $S$-functional calculus $\Phi$ for operators of type $\omega$
 $$
 \Phi:\Psi_L(\mathcal{S}^0_\mu) \ ({\rm resp.}\  \Psi_R(\mathcal{S}^0_\mu))   \to \mathcal{B}(V), \ \ \  \Phi: \phi\to \phi(T).
 $$
\item
  The commutative subalgebra of $\mathcal{F}_L(\mathcal{S}^0_\mu)$ consisting of intrinsic rational functions.
\item
The functions in $\mathcal{F}_L(\mathcal{S}^0_\mu)$ have at most polynomial growth.
So taken an intrinsic rational functions $\psi$ the operator $\psi(T)$ can be defined by the rational functional calculus. We assume also that  $\psi(T)$  is  injective.
\end{itemize}
\begin{definition}[$H^\infty$ functional calculus]
Let $V$ be a two-sided quaternionic Banach space and let $T\in\Omega$.
For $k\in \mathbb{N}$ consider the function
$$
\psi(s):=\Big(\frac{s}{1+s^2}\Big)^{k+1}.
$$
For $f\in \mathcal{F}_L(\mathcal{S}^0_\mu)$  we define the extended functional calculus as
\begin{equation}\label{HinftyFC}
f(T):=(\psi(T))^{-1}(\psi f)(T).
\end{equation}
 For $f\in \mathcal{F}_R(\mathcal{S}^0_\mu)$, and $T$ left linear we define the extended functional calculus as
\begin{equation}\label{HinftyFC2}
f(T):=( f\psi)(T)(\psi(T))^{-1}.
\end{equation}
We say that $\psi$ regularizes $f$.
\end{definition}

\begin{remark} In the previous definition
the operator $(\psi f)(T)$ (resp. $(f\psi )(T)$) is defined using the $S$-functional calculus $\Phi$ for operators of type $\omega$, and $\psi(T)$ is defined by the rational functional calculus.
\end{remark}
\begin{theorem}
The definition of the functional calculus  in (\ref{HinftyFC}) and in (\ref{HinftyFC2}) does not depend on the choice of the intrinsic rational slice hyperholomorphic function $\psi$.
\end{theorem}
\begin{proof} Let us prove (\ref{HinftyFC}).
Suppose that  $\psi$ and $\psi'$ are  two different regularizers and set
$$
A:=(\psi(T))^{-1}(\psi f)(T) \ \ \ {\rm and}\ \ \ \ B:=(\psi'(T))^{-1}(\psi' f)(T).
$$
Observe that since  the functions $\psi$ and $\psi'$ commute, because there are intrinsic rational functions, it is
$$
\psi(T)\psi'(T)=(\psi\psi')(T)=(\psi'\psi)(T)=\psi'(T)\psi(T),
$$
so we get
$$
(\psi'(T))^{-1}(\psi(T))^{-1}=(\psi(T))^{-1}(\psi'(T))^{-1}.
$$
It is now easy to see that
\[
\begin{split}
A=(\psi(T))^{-1}(\psi f)(T)&=(\psi(T))^{-1}(\psi'(T))^{-1}(\psi'(T))(\psi f)(T)=
\\
&
=(\psi'(T))^{-1}(\psi(T))^{-1} (\psi \psi' f)(T)
\\
&
=(\psi'(T))^{-1}(\psi(T))^{-1} \psi(T) (\psi' f)(T)
\\
&
=(\psi'(T))^{-1}(\psi' f)(T)=B,
\end{split}
\]
where we used the fact that from the product rule, see Proposition \ref{propratio},  we have that the inverse of  $\psi(T)$ is $(1/\psi)(T)$. The proof of \eqref{HinftyFC2} follows in a  similar way.
\end{proof}
 We now state an important result for functions in $\mathcal{F}_L(\mathcal{S}^0_\mu)$ (the same result with obvious changes  holds for functions in $\mathcal{F}_R(\mathcal{S}^0_\mu)$).
\begin{theorem} Let $f\in  \mathcal{F}(\mathcal{S}^0_\mu)$ and $g\in \mathcal{F}_L(\mathcal{S}^0_\mu)$.
Then we have
$$
f(T)+g(T)\subset (f + g)(T),
$$
$$
f(T)g(T)\subset (f g)(T),
$$
and $\mathcal{D}(f(T)g(T))=\mathcal{D}((f g)(T))\bigcap \mathcal{D}(g(T)) $.
\end{theorem}
\begin{proof}
Let us take $\psi_1$ and $\psi_2$ that regularize $f$ and $g$, respectively.
Observe that  the function $\psi:=\psi_1\psi_2$ regularize $f$, $g$, $f+g$ and $fg$ because $ \psi_1$, $\psi_2$ and $f$ commute among themselves.
Observe that
\[
\begin{split}
f(T)+g(T)&=(\psi(T))^{-1}(\psi f)(T)+(\psi(T))^{-1}(\psi g)(T)
\\
&
\subset (\psi(T))^{-1}[(\psi f)(T)+(\psi g)(T)]
\\
&
=(\psi(T))^{-1}[\psi (f+ g)](T)=(f+g)(T).
\end{split}
\]
We can consider now the product rule
\[
\begin{split}
f(T)g(T)&=(\psi_1(T))^{-1}(\psi_1 f)(T)\ (\psi_2(T))^{-1}(\psi_2 g)(T)
\\
&
\subset (\psi_1(T))^{-1}(\psi_2(T))^{-1}[(\psi_1 f)(T)(\psi_2 g)(T)]
\\
&
=(\psi_2(T)\psi_1(T))^{-1}[\psi_1(T)\psi_2(T) (f g)](T)
\\
&
=(\psi(T))^{-1}(\psi fg)(T)=(fg)(T),
\end{split}
\]
where we have used  $\psi:=\psi_1\psi_2$. Regarding the domains it is as in that complex case.
\end{proof}
\begin{remark}
{\rm We point out that in this case there is no spectral mapping theorem because the operator  $f(T)=(\psi(T))^{-1}(\psi f)(T)$ can be unbounded even when $f$ is bounded.}
\end{remark}
The following convergence theorem is stated for functions in $\mathcal{SH}_L^\infty(\mathcal{S}^0_\mu)$  but it holds also for functions in
 $\mathcal{SH}_R^\infty(\mathcal{S}^0_\mu)$ and is the quaternionic analogue of the theorem in Section 5 in \cite{McI1}.
The proof follows the proof of the convergence theorem in  \cite[p. 216]{McI1},
we just point out that the convergence theorem is based on the principle of uniform boundedness that
holds also for quaternionic operators.

 \begin{theorem}[Convergence theorem]\label{convth}
Suppose that $0\leq \omega <\mu\leq \pi$ and that $T$ is a linear operator of type $\omega$ such that it is one to one and with dense range.
Let $f_\alpha$ be a net in $\mathcal{SH}_L^\infty(\mathcal{S}^0_\mu)$ and let $f\in \mathcal{SH}_L^\infty(\mathcal{S}^0_\mu)$ and assume that:
\begin{itemize}
\item[(i)] There exists a positive constant $M$, such that $\|f_\alpha(T)\|\leq M$,
\item[(ii)]
For every $0<\delta<\lambda <\infty$
$$
\sup \{ |f_\alpha(s)-f(s)| \ \ {\rm such \ that}  \ \  s\in \mathcal{S}^0_\mu\ \ {\rm and} \ \ \delta\leq |s|\leq \lambda \}\to 0.
$$
\end{itemize}
Then $f(T)\in \mathcal{B}(V)$ and $f_\alpha(T)u\to f(T)u$ for all $u\in V$, moreover $\|f(T)\|\leq M$.
\end{theorem}
In the following section we discuss the boundedness of the $H^\infty$ functional calculus.

\section{Quadratic estimates and the $H^\infty$ functional calculus}
Let $\mathcal{H}$ be a right linear quaternionic Hilbert space with an $\mathbb{H}$-valued inner product $\langle \cdot, \cdot \rangle$ which satisfies, for every $\alpha$, $\beta\in \mathbb{H}$, and $x$, $y$, $z\in \mathcal{H}$, the relations:
\begin{align*}
\langle x, y \rangle =& \; \overline{ \langle y, x \rangle }, \\
\langle x, x \rangle \geq& \; 0 \quad {\rm and}
\quad \| x \|^2 := \langle x, x \rangle = 0 \Longleftrightarrow x = 0, \\
\langle x \alpha + y \beta, z \rangle =& \; \langle x, z \rangle \alpha + \langle y, z \rangle \beta, \\
\langle x, y \alpha + z \beta \rangle =& \; \bar{\alpha} \langle x, y \rangle + \bar{\beta} \langle x, z \rangle.
\end{align*}
\begin{definition}
We will call a subset $\mathcal{N} \subseteq \mathcal{H}$ a {\it Hilbert basis} if
\begin{align}
&\langle x, y \rangle = \; 0 \quad \text{for $x,y \in \mathcal{N}$ so that $x \neq y$,} \label{eq:Jun14r1} \\
&\langle x, x \rangle = \; 1 \quad \text{for $x \in \mathcal{N}$ so that $x \neq 0$,} \label{eq:Jun14r2} \\
&\{ x \in \mathcal{H}: \text{$\langle x, y \rangle = 0$ for all $y \in \mathcal{N}$}\} = \; \{ 0 \}.
\label{eq:Jun14r3}
\end{align}
\end{definition}
With a choice of the Hilbert basis $\mathcal{N}$ a quaternionic  Hilbert space on one side (left or right) can always been made two-sided. Thus it is not reductive to consider a quaternionic two-sided Hilbert space and repeat what we have done in the case of a Banach space to define a $H^\infty$ functional calculus.\\
 The crucial tool to show the boundedness of the $H^\infty$ functional calculus are the so called quadratic estimated, see \cite{McI1}.

\begin{definition}[Quadratic estimate]
Let $T$ be a right linear operator of type $\omega$ on a quaternionic Hilbert space $\mathcal{H}$ and let $\psi\in \Psi(\mathcal{S}^0_\mu)$ where $0\leq \omega <\mu\leq \pi$. We say that $T$ satisfies a quadratic estimate with respect to $\psi$ if there exists a positive constant $\beta$ such that
$$
\int_0^\infty\|\psi(tT)u\|^2\frac{dt}{t}\leq \beta^2\|u\|^2,\ \ \ {\rm for \ all} \ u\in \mathcal{H},
$$
where we write $\|u\|$ for $\|u\|_{\mathcal{H}}$.
\end{definition}
Let us introduce the notation
$$
\Psi^+(\mathcal{S}^0_\mu)=\{\psi\in \Psi(\mathcal{S}^0_\mu)\  : \ \psi(t)>0\ \ {\rm for \ all} \ t\in(0,\infty)\}
$$
and
$$
\psi_t(s)=\psi(ts),\ \ \ \ t\in(0,\infty).
$$
\begin{theorem} Let $0\leq \omega <\mu\leq \pi$ and assume that $T$ is a right linear operator in $\Omega$. Suppose that $T$ and its adjoint $T^*$ satisfy the quadratic estimates with respect to the functions $\psi$ and $\widetilde{\psi}\in \Psi^+(\mathcal{S}^0_\mu)$.
Suppose that $f$ belongs to $\mathcal{SH}_L^\infty(\mathcal{S}^0_\mu)$.
Then the operator $f(T)$ is bounded and there exists a positive constant $C$ such that
$$
\|f(T)\|\leq C\|f\|_\infty \ \ \ \ {\it for \ all} \ \ \ f\in \mathcal{SH}_L^\infty(\mathcal{S}^0_\mu).
$$
\end{theorem}
\begin{proof}
We follow the proof of Theorem at p. 221 in  \cite{McI1}, and  we point out the differences.
We observe that we choose the function $\psi$, $\widetilde{\psi}$ and $\eta$ in  the space of intrinsic functions
$\Psi^+(\mathcal{S}^0_\mu)$ because the pointwise product
$$
\varphi(s):=\psi(s)\widetilde{\psi}(s)\eta(s)
$$
has to be slice hyperholomorphic and moreover  $\eta$ has to be such that
$$
\int_0^\infty\varphi(t)\, \frac{dt}{t}=1.
$$
For $f\in \mathcal{SH}_L^\infty(\mathcal{S}^0_\mu)$ let us define
$$
f_{\varepsilon,R}(s)=\int_\varepsilon^R(\varphi_t f)(s)\frac{dt}{t}.
$$
Using the quadratic estimates it follows that
there exists a positive constant $C$ such that
$$
\|f_{ \varepsilon,R }(T)\| \leq C\|f\|_{\infty}
$$
the Convergence Theorem \ref{convth}  gives the formula
$$
f(T)u=\lim_{\varepsilon \to 0} \lim_{ R\to \infty}f_{\varepsilon,R}(T)u\ \ \  {\rm for\ all}\ \  u\in \mathcal{H}
$$
where $(\eta_t f)(T)$ is defined by the $S$-functional calculus
$$
(\eta_t f)(T)=\frac{1}{2\pi}\int_{\Gamma} S_L^{-1}(s,T)\, ds_i\, \eta_t(s)f(s),\ \ \ \ {\rm for\ all }\ \ \ f\in \Psi_L(\mathcal{S}^0_\mu),
$$
since $\eta_t f\in \Psi_L(\mathcal{S}^0_\mu)$ because  $\eta_t$ is intrinsic.
Precisely,
 the quadratic estimates and some computations show that there exists a positive constant $C_\beta$ such that
$$
|\langle f_{\varepsilon,R}(T)u, v\rangle|\leq C_\beta\sup_{t\in (0,\infty)}\|(\eta_tf)(T)\|\|u\|\|v\|.
$$
Since
\[
\begin{split}
\|(\eta_t f)(T)\|& \leq \frac{1}{2\pi} \|f\|_{\infty} \sup_{i\in \mathbb{S}}
\int_{\Gamma} \, \|S_L^{-1}(s,T)\|  |ds_i|\, |\eta_t(s)|
\\
&
 \leq \frac{1}{2\pi} \sup_{i\in \mathbb{S}}
 \|f\|_{\infty}\int_{\Gamma} \frac{C_\eta}{|s|}\, |ds_i|\, \frac{c|s|^\alpha}{1+|s|^{2\alpha}}
\\
&
\leq C_T(\mu,\eta) \|f\|_{\infty}.
\end{split}
\]
From the above estimates we get the statement.
\end{proof}

\section{The case of $n$-tuples of operators}\label{sec7}
The notion of slice hyperholomorphicity can be given for Clifford algebra-valued functions, see \cite{MR2752913}.
In this section, we recall the main results on this function theory and on the operators that we will consider later, without giving the details of the proofs (which can be found in \cite{MR2752913}).

\subsection{Preliminaries on the function theory}
Let $\rr_n$ be the real Clifford algebra over $n$ imaginary units $e_1,\ldots ,e_n$
satisfying the relations $e_ie_j+e_je_i=0$ for  $i\not= j$ and  $e_i^2=-1$.
 An element in the Clifford algebra will be denoted by $\sum_A e_Ax_A$ where
$A=\{ i_1\ldots i_r\}\in \mathcal{P}\{1,2,\ldots, n\},\ \  i_1<\ldots <i_r$
 is a multi-index
and $e_A=e_{i_1} e_{i_2}\ldots e_{i_r}$, $e_\emptyset =1$.
An element $(x_0,x_1,\ldots,x_n)\in \mathbb{R}^{n+1}$  will be identified with the element
$$
 x=x_0+\underline{x}=x_0+ \sum_{j=1}^nx_je_j,
$$
a so-called paravector,
in the Clifford algebra $\mathbb{R}_n$.
The real part $x_0$ of $x$ will also be denoted by $\Re(x)$.
The square of the norm of $x\in\mathbb{R}^{n+1}$ is defined by $|x|^2=x_0^2+x_1^2+\ldots +x_n^2$
 and the conjugate of $x$ is
$$
\bar x=x_0-\underline x=x_0- \sum_{j=1}^nx_je_j.
$$
Let
$$
\S:=\{ \underline{x}=e_1x_1+\ldots +e_nx_n\ : \  x_1^2+\ldots +x_n^2=1\}.
$$
Observe that for $i\in\S$, we obviously have $i^2=-1$.
Given an element $x=x_0+\underline{x}\in\rr^{n+1}$, we set
\[
i_x:=\begin{cases}\underline{x}/|\underline{x}| & \text{if }\underline{x}\neq0\\
\text{any element of $\S$} & \text{if } \underline{x} = 0,
\end{cases}
\]
then $x = u + i_xv$ with $u = x_0$ and $v = |\underline{x}|$.
For any element $x=u + i_xv\in\rr^{n+1}$, the set
$$
[x]:=\{y\in\rr^{n+1}\ :\ y=u+i v, \ i\in \mathbb{S}\}
$$

is an $(n-1)$-dimensional sphere in $\mathbb{R}^{n+1}$.
The vector space $\mathbb{R}+i\mathbb{R}$ passing through $1$ and
$i\in \mathbb{S}$ will be denoted by $\mathbb{C}_i$ and
an element belonging to $\mathbb{C}_i$ will be indicated by $u+iv$ with  $u$, $v\in \mathbb{R}$.

Since we identify the set of paravectors with the space $\mathbb{R}^{n+1}$, if $U\subseteq\mathbb{R}^{n+1}$ is an open set,
a function $f:\ U\subseteq \mathbb{R}^{n+1}\to\mathbb{R}_n$ can be interpreted as
a function of a paravector $x$.

\begin{definition}[Slice hyperholomorphic functions]
\label{defsmon}
Let $U\subseteq\mathbb{R}^{n+1}$ be an open set and let
$f: U\to\mathbb{R}_n$ be a real differentiable function. Let
$i\in\mathbb{S}$ and let $f_i$ be the restriction of $f$ to the
complex plane $\mathbb{C}_i$.
\\
The function  $f$ is said to be left slice  hyperholomorphic (or slice monogenic) if, for every
$i\in\mathbb{S}$, it satisfies
$$
\frac{1}{2}\left(\frac{\partial }{\partial u}f_i(u+iv)+i\frac{\partial
}{\partial v}f_i(u+iv)\right)=0
$$
 on $U\cap \mathbb{C}_i$. We denote the set of left slice hyperholomorphic functions on the open set $U$  by
$\mathcal{SM}_L(U)$.
\\
The function $f$ is said to be right slice hyperholomorphic (or right slice monogenic) if,
for every
$i\in\mathbb{S}$, it satisfies
$$
\frac{1}{2}\left(\frac{\partial }{\partial u}f_i(u+iv)+\frac{\partial
}{\partial v}f_i(u+iv)i\right)=0
$$
on $U\cap \mathbb{C}_i$.
We denote the set of right slice hyperholomorphic functions on the open set $U$  by $\mathcal{SM}_R(U)$.

A left (or right) slice  hyperholomorphic function that satisfies $f(U\cap\C_i)\subset\C_i$ for any $i\in\S$ is called intrinsic. The set of all intrinsic functions will be denoted by $\intrin(U)$.
\end{definition}
\begin{remark}
We use the same symbol $\intrin(U)$ to denote intrinsic functions
for the quaternionic case and for the Clifford algebra case,  the meaning is clear from the context
and no confusion arises.
\end{remark}
\begin{remark} Let $x$ be a paravector, then
any power series of the form $\sum_{\ell=0}^\infty x^\ell a_\ell$ with $a_\ell\in\mathbb{R}_n$, for $\ell\in \mathbb{N}$, is left slice hyperholomorphic and any power series of the form $\sum_{\ell=0}^\infty b_\ell x^\ell$ with $b_\ell\in\mathbb{R}_n$, for $\ell\in \mathbb{N}$, is right slice hyperholomorphic.
In the case $a_\ell$, or similarly $b_\ell$ (for all $\ell\in \mathbb{N}\cup\{0\}$) are real numbers the power series define intrinsic functions, where they converge. They are both left and right slice hyperholomorphic.
\end{remark}

\begin{lemma}[Splitting Lemma]\label{SplitLemM}
Let $U\subset\mathbb{R}^{n+1}$ be open and let $f:U\to\mathbb{R}_n$ be a left slice hyperholomorphic function. For every $i=i_1\in\S$, let $i_2,\ldots,i_n$ be a completion to a basis of $\mathbb{R}_n$ that satisfies the relation  $i_ri_s + i_s i_r = - 2\delta_{r,s}$.
Let $f_i$ be the restriction of $f$ to the complex plane $\mathbb{C}_i$. Then there exist $2^{n-1}$ holomorphic functions $F_A: U\cap\C_i\to\C_i$ such that for every $z\in U\cap\C_i$
\[f_i (z) = \sum_{|A|=0}^{n-1}F_A(z)i_A\]
where $i_A = i_{\ell_1}\cdots i_{\ell_s}$ for any nonempty subset $A=\{\ell_1<\ldots<\ell_s\}$  of $\{2,\ldots,n\}$ and $i_{\emptyset} = 1$.

Similarly, if $g:U\to\H$ is right slice hyperholomorphic, then there exist $2^{n-1}$ holomorphic functions $G_A:U\cap\C_i\to\C_i$ such that for every $z\in U\cap\C_i$
\[g_i(z) = \sum_{|A|=0}^{n-1}i_AG_A(z).\]
\end{lemma}

Slice hyperholomorphic functions have good properties when they are defined on suitable domains whose definition mimics the one in the quaternionic case.

\begin{definition}[Axially symmetric slice domain]\label{axsymm}
Let $U$ be a domain in $\rr^{n+1}$.
We say that $U$ is a
\textnormal{slice domain} (s-domain for short) if $U \cap \mathbb{R}$ is nonempty and if $U\cap \mathbb{C}_i$ is a domain in $\mathbb{C}_i$ for all $i \in \mathbb{S}$.
We say that $U$ is
\textnormal{axially symmetric} if, for all $x \in U$, the
$(n-1)$-sphere $[x]$ is contained in $U$.
\end{definition}
The crucial result of slice hyperholomorphic functions is the representation formula (or structure formula), first proved in \cite{TREND}.

\begin{theorem}[Representation Formula]\label{formulamonM}
Let $U\subseteq
\mathbb{R}^{n+1}$ be an axially symmetric s-domain and let
$f\in \mathcal{SM}_L(U)$. Then
 for any  vector
$x =u+i_x v\in U$ the following formula hold:
\begin{equation}\label{distributionmon1}
f(x)=\frac{1}{2}\Big[   f(u +i v)+f(u -i v)\Big] +\frac{1}{2}\Big[i_x
i[f(u -i v)-f(u +i v)]\Big].
\end{equation}
If $f\in \mathcal{SM}_R(U)$ then
\begin{equation}\label{distributionmon1}
f(x)=\frac{1}{2}\Big[   f(u +i v)+f(u -i v)\Big] +\frac{1}{2}\Big[
[f(u -i v)-f(u +i v)]i\Big]i_x.
\end{equation}
\end{theorem}
The proof of the following Cauchy formula
is based on the Representation Formula.
\begin{theorem}[The Cauchy formula with slice hyperholomorphic kernel]
\label{Cauchygenerale}
Let $U\subset\mathbb{R}^{n+1}$ be an axially symmetric s-domain.
Suppose that $\partial (U\cap \mathbb{C}_i)$ is a finite union of
continuously differentiable Jordan curves  for every $i\in\mathbb{S}$ and set  $ds_i=-i\, ds $.
\\
If $f$ is
a (left) slice hyperholomorphic function on a set that contains $\overline{U}$ then
\begin{equation}\label{Cauchyleft}
 f(x)=\frac{1}{2 \pi}\int_{\partial (U\cap \mathbb{C}_i)} S_L^{-1}(s,x)\,ds_i\, f(s)
\end{equation}
where
\begin{equation}\label{SL1}
S_L^{-1}(s,x):=-(x^2 -2 \Re(s)x+|s|^2)^{-1}(x-\overline s),\ \ \ x\not\in[s].
\end{equation}
\\
If $f$ is a right slice hyperholomorphic function on a set that contains $\overline{U}$,
then
\begin{equation}\label{Cauchyright}
 f(x)=\frac{1}{2 \pi}\int_{\partial (U\cap \mathbb{C}_i)}  f(s)\,ds_i\, S_R^{-1}(s,x)
 \end{equation}
 where
 \begin{equation}\label{SR1}
S_R^{-1}(s,x):= -(x-\bar s)(x^2-2{\rm Re}(s)x+|s|^2)^{-1}, \ \ \ \ x\not\in[s].
\end{equation}
The
integrals  depend neither on $U$ nor on the imaginary unit
$i\in\mathbb{S}$.
\end{theorem}

\subsection{Preliminaries on operator theory}

In the following we will assume that $V$ is a real Banach space.
We denote by  $V_n$  the two-sided Banach module  over $\mathbb{R}_n$ corresponding to $V\otimes \mathbb{R}_n$.
For $\mu=0,1,...,n$ let  $T_\mu :  \mathcal{D}(T_\mu) \to V$ be linear operators that do not necessarily commute among themselves, where $\mathcal{D}(T_\mu)$ denotes the domain of $T_\mu$ which is contained in $V$.
To define the $H^\infty$ functional calculus for the $(n+1)$-tuple of linear operators $(T_0,T_1,..,T_n)$,
we will consider the so-called operator in paravector form
$$
T=T_0+e_1T_1+\ldots +e_nT_n
$$
whose domain is $\mathcal{D}(T)=\bigcap_{\mu=0}^n \mathcal{D}(T_\mu)$.

The case of  $n$-tuples of operators is obviously contained in the previous case by setting $T_0=0$, namely when we take  $(0,T_1,..,T_n)$. In the following we will consider
$n+1$-tuples of operators, including operator $T_0$, because from the point of view of our theory there are no additional difficulties.

We denote by $\mathcal{B}(V)$ the space
of all bounded $\mathbb{R}$-homomorphisms from the Banach space $V$ into itself
 endowed with the natural norm denoted by $\|\cdot\|_{\mathcal{B}(V)}$.
 Let $T_A\in \mathcal{B}(V)$. We define the operator
$$T=\sum_A T_Ae_A$$
 and
its action on the generic element of $V_n$
$$
v=\sum_B v_Be_B
$$
as
$$T(v)=\sum_{A,B}
T_A(v_B)e_Ae_B.
$$
 The operator $T$ is a right-module homomorphism which is a bounded linear
map on $V_n$.
The set of all such bounded operators is denoted by $\mathcal{B}(V_n)$.
We define a norm in $\mathcal{B}(V_n)$ by setting
$$
\|T\|_{\mathcal{B}(V_n)}=\sum_A \|T_A\|_{\mathcal{B}(V)}.
$$
We denote by $\mathcal{K}(V_n)$ the set of those paravector operators $T$ that are linear and closed.
The notion of $S$-spectrum, $S$-resolvent set and of $S$-resolvent operator can be defied in the Clifford setting as follows.
\begin{definition}
The $S$-resolvent set $\rho_S(T)$ of $T$ is defined as
$$
\rho_S(T):=\{ s\in \mathbb{R}^{n+1}\  : \ Q_{s}(T)\in \mathcal{B}(V_n)\}
$$
where
\begin{equation}\label{defQmon}
Q_{s}(T):=(T^2-2 {\rm Re}(s) T +|s|^2\mathcal{I})^{-1},
\end{equation}
and
$$
T^2-2 {\rm Re}(s) T +|s|^2\mathcal{I}\, :\, \mathcal{D}(T^2)\to V_n.
$$
The $S$-spectrum $\sigma_S(T)$ of $T$ is defined by
$$
\sigma_S(T)= \mathbb{R}^{n+1}\setminus \rho_S(T).
$$
\end{definition}
In the case of $n$-tuples of operators when $T$ is bounded, i.e. all the components $T_\mu$ are bounded, then the $S$-spectrum is a nonempty compact set in $\mathbb{R}^{n+1}$. When at least one of the operators $T_\mu$ is unbounded then $\rho_S(T)$ can be every closed subset of $\mathbb{R}^{n+1}$. In this case, we will always assume that the set $\rho_S(T)$ is nonempty.
\begin{definition}
Let $T\in\closOP(V_n)$. The left $S$-resolvent operator is defined as
\begin{equation}\label{SresolvoperatorL}
S_L^{-1}(s,T):= Q_s(T)\overline{s} -TQ_s(T),\ \ \ \ s\in \rho_S(T)
\end{equation}
and the right $S$-resolvent operator is defined as
\begin{equation}\label{SresolvoperatorR}
S_R^{-1}(s,T):=-(T-\id \overline{s})Q_s(T),\ \ \ \ s\in \rho_S(T).
\end{equation}
\end{definition}

The  $S$-resolvent equation for the case of unbounded operators $T$ has been considered in \cite{CG} for the quaternionic operators but its proof can be easily adapted to the the case of $n$-tuples of operators.
Let $T\in\closOP(V_n)$. For  $s,p \in  \rho_S(T)$ with $s\notin[p]$, it is
\begin{equation}
\begin{split}
S_R^{-1}(s,T)S_L^{-1}(p,T)v=&\big[[S_R^{-1}(s,T)-S_L^{-1}(p,T)]p
\\
&
-\overline{s}[S_R^{-1}(s,T)-S_L^{-1}(p,T)]\big](p^2-2s_0p+|s|^2)^{-1}v, \ \ \ v\in V_n.
\end{split}
\end{equation}

\subsection{The rational functional calculus}
In the Clifford algebra setting, the definition of slice hyperholomorphic rational function is slightly more complicated than in the quaternionic case. To introduce it, we need some preliminary definitions and results.\\
We begin by defining the slice hyperholomorphic product,
 which is more involved than in the quaternionic case.
\\
\\
{\em The slice hyperholomorphic product}.
\\
For any $i\in\mathbb{S}$ set $i=i_1$ and
 consider a completion
 to a basis $\{i_1,\ldots, i_n\}$ of $\mathbb{R}_{n}$ such that $i_\ell i_{\ell'}+i_{\ell'}i_\ell=-2\delta_{\ell\ell'}$.
 The Splitting Lemma \ref{SplitLemM} guarantees the existence of holomorphic functions
 $F_A,G_A: \ U\cap \mathbb{C}_i\to \mathbb{C}_i$ such
 that for all $z=u+iv\in  U\cap \mathbb{C}_i$, the restriction to $\mathbb{C}_i$ of $f$ and $g$, denoted by $f_i$ and $g_i$ respectively, can be written as
 $$
 f_i(z)=\sum_A F_A(z)i_A, \qquad g_i(z)=\sum_B G_B(z)i_B,
 $$
 where $A,B$ are subsets of $\{2,\ldots ,n\}$ and, by definition,
 $i_\emptyset =1$.
We define the function $f_i*g_i:\ U\cap \mathbb{C}_i\to \mathbb{R}_n$  as
\begin{equation}\label{f*gmonM}
(f_i*g_i)(z)=\sum_{|A|{\rm even}}(-1)^{\frac{|A|}{2}}F_A(z)G_A(z)
+\sum_{|A|{\rm odd}}(-1)^{\frac{|A|+1}{2}}F_A(z)\overline{G_A(\bar
z)}
$$
$$
+\sum_{|A|{\rm even},B\not=A}F_A(z)G_B(z)i_Ai_B +\sum_{|A|{\rm
odd},B\not=A}F_A(z)\overline{G_B(\bar z)}i_Ai_B.
\end{equation}
Then  $(f_i*g_i)(z)$ is obviously a  holomorphic map on $\mathbb{C}_i$,
and hence its
unique slice hyperholomorphic extension to $U$, which can be constructed according to the
Representation Formula \ref{formulamonM},
is given by
$$
(f*g)(x):={\rm ext}(f_i*g_i)(x).
$$

 \begin{definition}
Let $U\subseteq\mathbb{R}^{n+1}$ be an axially symmetric s-domain and let
 $f,g:\ U\to\mathbb{R}_n$  be  left slice hyperholomorphic functions. The function
$$(f*g)(x)={\rm ext}(f_i*g_i)(x)$$ defined as the extension of (\ref{f*gmonM})
is called the s-monogenic product  of $f$ and $g$. This product is
called $*$-product of $f$ and $g$.
 \end{definition}
An analogous definition can be made for right slice hyperholomorphic functions. Since we will concentrate on
intrinsic rational functions in the sequel, which are both left and right slice hyperholomorphic we will limit ourselves to the left case.
\\
\\
{\em The $*$-inverse function}.
\\
 Let $U\subseteq\mathbb{R}^{n+1}$ be an axially symmetric s-domain and let
 $f:\ U\to\mathbb{R}_{n}$ be a left slice hyperholomorphic function.
Let us consider the restriction $f_i(z)$ of $f$ to the plane $\mathbb{C}_i$ and its  representation given by the Splitting Lemma
 \begin{equation}\label{monfi}
 f_i(z)=\sum_A F_A(z)i_A.
 \end{equation}
Let us define the function $f_i^c:\ U\cap \mathbb{C}_i\to \mathbb{C}_i$ as
\[
\begin{split}
f_i^c(z):&=\sum_A F_A^c(z)i_A
\\
&
=\sum_{|A|\equiv  0} \overline{F_A(\bar z)}i_A-\sum_{|A|\equiv 1} F_A(z)i_A -\sum_{|A|\equiv  2}  \overline{F_A(\bar z)}i_A+\sum_{|A|\equiv  3} F_A(z)i_A,
\end{split}
\]
where the equivalence $\equiv$ is intended as $\equiv ({\rm mod}
4)$, i.e. the congruence modulo $4$. Since  any function $F_A$ is obviously holomorphic
it can be uniquely extended to a slice hyperholomorphic function on $U$,
according to the Representation Formula. Thus we can give the following definition:
 \begin{definition}\label{f^cmon}
Let $U\subseteq\mathbb{R}^{n+1}$ be an axially symmetric s-domain and let
 $f:\ U\to\mathbb{R}_n$  be a left slice hyperholomorphic function. The function
$$
f^c(x)={\rm ext}(f_i^c)(x)
$$
is called the slice hyperholomorphic conjugate of $f$.
 \end{definition}

Using the notion of $*$-multiplication of slice hyperholomorphic functions,
it is possible to associate to any slice hyperholomorphic function $f$ its
{\em symmetrization}  denoted by $f^s$.
Let $U\subseteq\mathbb{R}^{n+1}$ be an axially symmetric s-domain, let
$f:\ U\to\mathbb{R}_{n}$ be a slice hyperholomorphic function, and let is restriction to $\mathbb C_i$ be as in \eqref{monfi}.
Here we will use the notation $[f_i]_0$ to denote the ``scalar" part of
the function $f_i$, i.e. the part whose coefficient in the
Splitting Lemma is $i_\emptyset=1$. We define
the function $f^s:\ U\cap \mathbb{C}_i\to \mathbb{C}_i$ as
\[
\begin{split}
f^s_i:&=[f_I*f^c_i]_0
\\
&
=\Big[(\sum_BF_B(z)I_B)(\sum_{|A|\equiv  0} \overline{F_A(\bar
z)}I_A -\sum_{|A|\equiv  1} F_A(z)I_A
-\sum_{|A|\equiv  2}  \overline{F_A(\bar z)}I_A+\sum_{|A|\equiv  3}
F_A(z)I_A)\Big]_0.
\end{split}
\]
The function $f_i^s$ is  holomorphic and hence its unique
slice hyperholomorphic extension to $U$ motivates the following definition:

\begin{definition}
Let $U\subseteq\mathbb{R}^{n+1}$ be an axially symmetric s-domain and let
 $f:\ U\to\mathbb{R}_{n}$  be a slice hyperholomorphic function. The function
$$
f^s(x)={\rm ext}(f_i^s)(x)
$$
defined by the extension of
$f^s_i=[f_i*f^c_i]_0$ from $U\cap \mathbb{C}_i$ to the whole $U$ is called the symmetrization  of $f$.
\end{definition}
The following lemma is important for the definition of the $*$-inverse.
\begin{lemma}
Let $U\subseteq\mathbb{R}^{n+1}$ be an axially symmetric s-domain and let
$f,g$ be left slice hyperholomorphic functions. Then
$$
f^sg=f^s*g=g*f^s.
$$
Moreover, if $Z_{f^s}$ is the zero set of $f^s$, then
$$
(f^s)^{-1}g=(f^s)^{-1}*g=g*(f^s)^{-1}\ \ {\it
on}\ U\setminus Z_{f^s}.
$$
\end{lemma}

\begin{definition}\label{inverse_mon}
Let $U\subseteq \mathbb{R}^{n+1}$ be an axially symmetric s-domain.  Let $f:
U\to\mathbb{R}_n$ be a left slice hyperholomorphic function such that for some
$i\in\mathbb{S}$ its restriction $f_i$ to the complex plane $\mathbb{C}_i$ satisfies the condition
\begin{equation}\label{propertyi}
f_i*f_i^c\ \ {\it has\ values\ in\ }\mathbb{C}_i.
\end{equation}
We define the function:
$$
f^{-*}:={\rm ext}((f^s_i)^{-1}f^c_i)
$$
where $f_i^s=[f_i*f_i^c]_0=f_i*f_i^c$, and we will call it slice hyperholomorphic inverse of the function $f$.
\end{definition}
The next proposition shows that the function  $f^{-*}$ is the  inverse of $f$ with respect to the $*$-product:
\begin{lemma}
Let $U\subseteq \mathbb{R}^{n+1}$ be an axially symmetric s-domain. Let $f:
U\to\mathbb{R}_n$ be an s-monogenic function such that for some
$i\in\mathbb{S}$ we have that
$f_i*f_i^c$ has values in $\mathbb{C}_i.$
Then on $U\setminus Z_{f^s}$ we have:
$$f^{-*}*f=f*f^{-*}=1.
$$
\end{lemma}
\begin{remark}{\rm Note that the $*$-inverse of a slice hyperholomorphic function $f$ is defined under the additional assumption that that for some
$i\in\mathbb{S}$ we have that
$f_i*f_i^c$ has values in $\mathbb{C}_i$.
This assumption is automatically satisfied, for all $i\in\mathbb S$, by the intrinsic functions.}
\end{remark}
{\em The rational functional calculus}.
\\
Consider a left slice hyperholomorphic  polynomial
$$
\mathcal{P}(x)=\sum_{\ell=0}^m x^\ell a_\ell,\ \ \ {\rm where }\ \ a_\ell\in \mathbb{R}_n
$$
in the paravector variable $x$. The natural functional calculus
is obtained by replacing the paravector operator $x$ by the paravector operator $T=T_0+T_1e_1+\ldots +T_ne_n$:
$$
\mathcal{P}(T)=\sum_{\ell=0}^m T^\ell a_\ell,\ \ \ {\rm where }\ \ a_\ell\in \mathbb{R}_n
$$
whose domain is $\mathcal{D}(\mathcal{P}(T))=D(T^m)$.
\\
\\
Let $\mathcal{Q}(s)$ be a polynomial in the paravector variable $s$ satisfying the condition \eqref{propertyi}, i.e.
$\mathcal{Q}_i*\mathcal{Q}_i^c$ has values in $\mathbb{C}_i$ for every  $i\in\mathbb{S}$.
Then  $\mathcal{Q}^{-*}(s)$ is a rational function.
If we use the $*$-multiplication and if $\mathcal{P}(s)$ is a polynomial then rational functions are of the form
$$
\mathcal{R}(s)=\mathcal{P}(s)* \mathcal{Q}^{-*}(s)
$$
or
$$
\mathcal{R}(s)=\mathcal{Q}^{-*}(s)*\mathcal{P}(s).
$$
In the sequel we will be interested in the functional calculus for intrinsic rational functions, so
$$
\mathcal{P}(s)* \mathcal{Q}^{-*}(s)=\mathcal{Q}^{-*}(s)*\mathcal{P}(s)=\dfrac{\mathcal{P}(s)}{\mathcal{Q}(s)}.
$$
Let $\mathcal{R}$ be a rational function  and assume
that $\mathcal{R}$ has no poles on the S-spectrum of $T$; suppose that $T$ is a closed densely defined paravector operator and define
$$
\mathcal{R}(T)=\mathcal{P}(T)* \mathcal{Q}^{-*}(T).
$$
This operator is also closed and densely defined and its domain is $\mathcal{D}(T^m)$ where
$$
m:=\max\{0, \deg P-\deg Q\}.
$$
We point out the Propositions \ref{RATUNO} and \ref{propratio} holds also in this setting and we do not repeat them.

\subsection{The $S$-functional calculus for $n$-tuples of operators of type $\omega$}

The $S$-resolvent operators, on which the $S$-functional calculus for $n$-tuples of operators is based, are
defined as follows:
for  $T\in\closOP(V)$ the left $S$-resolvent operator is
\begin{equation}\label{SresolvoperatorL}
S_L^{-1}(s,T):= Q_s(T)\overline{s} -TQ_s(T),\ \ \ \ s\in\rho_S(T),
\end{equation}
and the right $S$-resolvent operator is
\begin{equation}\label{SresolvoperatorR}
S_R^{-1}(s,T):=-(T-\id \overline{s})Q_s(T),\ \ \ \ s\in\rho_S(T),
\end{equation}
where
$$
Q_s(T):=(T^2 - 2\Re(s)T + |s|^2\id)^{-1}.
$$

The argument function for $p\in\mathbb{R}^{n+1}\setminus\{0\}$ is denoted by $\arg(p)$ and it is
 the unique number $\theta \in [0,\pi]$ such that $p = |p| e^{\theta i_p}$.
Observe that $\theta = \arg(s)$ does not depend on the choice of $i_s$ if $s\in\rr\setminus\{0\}$ since $p = |p|e^{0 i}$ for any $i\in\S$ if $p>0$ and $p = |p|e^{\pi i}$ for any $i\in\S$ if $p<0$.

Let $\vartheta \in [0, \pi]$ and us define the sets
$$
\mathcal{S}_\vartheta=\{ s\in \mathbb{R}^{n+1}\ :\ s=0\ {\rm or} \ |\arg(p)|\leq \vartheta\}
$$
$$
\mathcal{S}^0_\vartheta=\{ s\in \mathbb{R}^{n+1}\ :\ \ |\arg(p)|< \vartheta\}.
$$

\begin{definition}[Paravector operators of type $\omega$]
Let $\omega\in [0,\pi)$ we say that the paravector operator $T:\mathcal{D}(T)\subseteq V_n\to V_n$, where $T=T_0+T_1e_1+\ldots +T_ne_n$,  is of type $\omega$ if
\begin{itemize}
\item[(i)]
$T$ is closed and densely defined,
\item[(ii)]
$\sigma_S(T)\subset \mathcal{S}_\vartheta\cup\{\infty\}$,
\item[(iii)]
for every $\vartheta\in (\omega,\pi]$ there exists a positive constant $C_\vartheta$ such that
$$
\|S_L^{-1}(s,T)\|\leq \frac{C_\vartheta}{|s|}\ \ {\it for  \ all  \ non\ zero}\  s\in\mathcal{S}^0_\vartheta,
$$
$$
\|S_R^{-1}(s,T)\|\leq \frac{C_\vartheta}{|s|}\ \ {\it for  \ all  \ non\ zero}\  s\in\mathcal{S}^0_\vartheta.
$$
\end{itemize}
\end{definition}
We now define suitable subsets of the set of slice hyperholomorphic functions.
Let $\mu\in (0,\pi]$. We set
$$
\mathcal{S\!M}_L^\infty(\mathcal{S}^0_\mu)=\{f\in \mathcal{S\!M}_L(\mathcal{S}^0_\mu)\ \ \ {\it such \ that}\ \ \|f\|_\infty:=\sup_{s\in\mathcal{S}^0_\mu}|f(s)|<\infty \},
$$
$$
\mathcal{S\!M}_R^\infty(\mathcal{S}^0_\mu)=\{f\in \mathcal{S\!M}_R(\mathcal{S}^0_\mu)\ \ \ {\it such \ that}\ \ \|f\|_\infty:=\sup_{s\in\mathcal{S}^0_\mu}|f(s)|<\infty \},
$$
$$
\mathcal{N}^\infty(\mathcal{S}^0_\mu):=\{f\in \mathcal{N}(\mathcal{S}^0_\mu)\ \ \ {\it such \ that}\ \ \|f\|_\infty:=\sup_{s\in\mathcal{S}^0_\mu}|f(s)|<\infty\}.
$$
 To define the $S$-functional calculus for paravector operators of type $\omega$ we need the following definition.
\begin{definition} Let $\mu\in (0,\pi]$, we set
$$
\Psi_L(\mathcal{S}^0_\mu)=\{f\in \mathcal{S\!M}_L^\infty(\mathcal{S}^0_\mu)\ {\it such \ that} \ \exists \ \alpha>0,\ c>0\ \
|f(s)|\leq \frac{c|s|^\alpha}{1+|s|^{2\alpha}},\ \ {\it for \ all}\ s\in \mathcal{S}^0_\mu\},
$$
$$
\Psi_R(\mathcal{S}^0_\mu)=\{f\in \mathcal{S\!M}_R^\infty(\mathcal{S}^0_\mu)\ {\it such \ that} \ \exists \ \alpha>0,\ c>0\ \
|f(s)|\leq \frac{c|s|^\alpha}{1+|s|^{2\alpha}},\ \ {\it for \ all}\ s\in \mathcal{S}^0_\mu\},
$$
$$
\Psi(\mathcal{S}^0_\mu)=\{f\in \mathcal{N}^\infty(\mathcal{S}^0_\mu)\ {\it such \ that} \ \exists \ \alpha>0,\ c>0\ \
|f(s)|\leq \frac{c|s|^\alpha}{1+|s|^{2\alpha}},\ \ {\it for \ all}\ s\in \mathcal{S}^0_\mu\}.
$$
\end{definition}

The analogue of Theorem \ref{theoindep} in  the Clifford algebra setting allows to define the $S$-functional calculus for paravector operators of type $\omega$. Thus, with the identification of the $(n+1)$-tuple $(T_0,T_1,...,T_n)$ of linear operators with the paravector operator  $T=T_0+T_1e_1+\ldots +T_ne_n$, we obtain the $S$-functional calculus for $n+1$-tuples of operators and, as we discussed, also for $(0,T_1,...,T_n)$.
\begin{definition}[The $S$-functional calculus for $n$-tuples of operators of type $\omega$]\label{psiT}
Let $i\in\S$, and let $\mathcal{S}^0_\mu$ be the sector defined above. Choose a piecewise smooth path $\Gamma$ in $\mathcal{S}^0_\mu\cap\C_i$ that goes from $\infty e^{i\theta}$ to $\infty e^{-i\theta}$, for $\omega <\theta <\mu$,  then
\begin{equation}\label{PsiTLmon}
\psi(T):= \frac{1}{2\pi}\int_{\Gamma}S_L^{-1}(s,T) \, ds_i\, \psi(s),\ \ \ \ {\it for\ all }\ \ \ \psi\in \Psi_L(\mathcal{S}^0_\mu),
\end{equation}\begin{equation}\label{PsiTRmon}
\psi(T):= \frac{1}{2\pi}\int_{\Gamma} \psi(s)\, ds_i\, S_R^{-1}(s,T),\ \ \ \ {\it for\ all }\ \ \ \psi\in \Psi_R(\mathcal{S}^0_\mu).
\end{equation}
\end{definition}
The definition is well posed since the two integrals above do not depend neither on $\Gamma$ nor on $i\in \mathbb{S}$.
\\
If $\psi$  belongs to $\Psi(\mathcal{S}^0_\mu)$ both the representations \eqref{PsiTLmon} and \eqref{PsiTRmon} can be used and
$$
\psi(T):= \frac{1}{2\pi}\int_{\Gamma} \psi(s)\, ds_i\, S_R^{-1}(s,T)=\frac{1}{2\pi}\int_{\Gamma}S_L^{-1}(s,T) \, ds_i\, \psi(s),
\ \ \ \ {\rm for\ all }\ \ \ \psi\in \Psi(\mathcal{S}^0_\mu).
$$
As in the quaternionic case, the $S$-functional calculus satisfies the following properties.
\begin{theorem}
The operators $\psi(T)$ defined in (\ref{PsiTLmon}) and (\ref{PsiTRmon}) are bounded linear operators:
 $$(a\psi+b\varphi)(T)=a\psi(T)+b\varphi(T),\ \ \ \ {\it for \ all}\ \ \ \psi,\varphi\in \Psi_L(\mathcal{S}^0_\mu),$$
$$(a\psi+b\varphi)(T)=a\psi(T)+b\varphi(T),\ \ \ \ {\it for \ all}\ \ \ \psi,\varphi\in \Psi_R(\mathcal{S}^0_\mu).$$
and moreover
$$ (\psi\varphi)(T)=\psi(T)\varphi(T),\ \ \ \ {\it for \ all}\ \ \ \psi\in \Psi(\mathcal{S}^0_\mu),\ \ \ \varphi\in \Psi_L(\mathcal{S}^0_\mu),
$$
$$ (\psi\varphi)(T)=\psi(T)\varphi(T),\ \ \ \ {\it for \ all}\ \ \ \psi\in \Psi_R(\mathcal{S}^0_\mu),\ \ \ \varphi\in \Psi(\mathcal{S}^0_\mu).$$
\end{theorem}

\subsection{The $H^\infty$ functional calculus}
We are now in the position to state the  $H^\infty$ functional calculus for $n+1$-tuples of operators:
\begin{definition}
Let $\omega$ and $\mu$ be two real numbers such that $0\leq \omega <\mu\leq \pi$ and we make the following assumptions on the linear operator $T$
\begin{itemize}
\item[(i)] $T$ is an operator of type $\omega$;
\item[(ii)] $T$ is one-to-one and with dense range.
\end{itemize}
We define
$$
\mathcal{F\!M}_L(\mathcal{S}^0_\mu)=\{f\in \mathcal{S\!M}_L(\mathcal{S}^0_\mu)\ \ \ {\rm such \ that}\
|f(s)|\leq C(|s|^k+|s|^{-k})\ \ {\rm for\ some}\ \ k >0\  {\rm and}\  C>0 \},
$$
$$
\mathcal{F\!M}_R(\mathcal{S}^0_\mu)=\{f\in \mathcal{S\!M}_R(\mathcal{S}^0_\mu)\ \ \ {\rm such \ that}\
|f(s)|\leq C(|s|^k+|s|^{-k})\ \ {\rm for\ some}\ \ k>0\  {\rm and}\  C>0 \}.
$$
$$
\mathcal{F\!M}(\mathcal{S}^0_\mu)=\{f\in \mathcal{N\!M}(\mathcal{S}^0_\mu)\ \ \ {\rm such \ that}\
|f(s)|\leq C(|s|^k+|s|^{-k})\ \ {\rm for\ some}\ \ k>0\  {\rm and}\  C>0 \}.
$$
\end{definition}
\begin{definition}
For $k>0$ let us set
$$
\psi(s):=\Big(\frac{s}{1+s^2}\Big)^{k+1}.
$$
For $f\in \mathcal{F\!M}_L(\mathcal{S}^0_\mu)$  we define the functional calculus
$$
f(T):=(\psi(T))^{-1}(\psi f)(T),
$$
where the operator $(f\psi)(T)$ is defined using the $S$-functional calculus, and $\psi(T)$ is defined by the rational functional calculus.
\end{definition}
The independence of the definition from the regularizing function $\psi$ and the product rule holds also here.

\begin{remark}
{\rm We point out that in this case there is no spectral mapping theorem available and the operator  $f(T)=(\psi(T))^{-1}(f\psi)(T)$ can be unbounded also when $f$ is bounded.}
\end{remark}
The analogue of the convergence theorem \ref{convth} holds in the Clifford algebra case, and it can be considered the paravector case version of the theorem in Section 5 in \cite{McI1}.
\\
\\
\begin{remark}{\rm
We conclude by pointing out that some
examples of operators to which the $S$-functional and the $H^\infty$-functional calculi apply are the Dirac operator
 $$
 D=e_1\frac{\partial }{\partial x_1}+\ldots+e_n\frac{\partial }{\partial x_n}
 $$
 and the global operator that annihilates slice hyperholomorphic functions, see \cite{Global}:
 $$
G(x)=|\underline{x}|^2\frac{\partial }{\partial x_0}  \ +  \  \underline{x}  \  \sum_{j=1}^n  x_j\frac{\partial }{\partial x_j}.
$$
}
\end{remark}
\begin{remark}
{\rm
In a Hilbert space a quadratic estimate is essentially contained in the Plancherel theorem.  On Lipschitz perturbations of the classical spaces, including the real line and the Euclidean spaces, there are no Plancherel theorems, and the boundedness is technically difficult.
In later contexts the quadratic estimates are consequences of the Coifman-McIntosh-Meyer (CMcM) Theorem. In fact, people found direct proofs, not via quadratic estimates.

The papers \cite{2LiMISe} and \cite{3} are equivalent, providing two simplest, independent and direct proofs  of the boundedness of the $H^\infty$ functional calculus of the Clifford Dirac differential operators $\underline{D}=D_1e_1+\cdots +D_ne_n$ and the non-homogeneous
$D=D_0+\underline{D},$ where $D_k, k=0,1,...,n$ is the partial differential operator with respect to the $k$-th variable.

The paper \cite{4} deals with the equivalence relationship between the three forms of the functional calculus, viz., the Cauchy-Dunford form, the Fourier multiplier form under the Fourier transformation theory on Lipschitz curves and surfaces of Coifman and Meyer, and the monogenic singular integral operator form.  Indeed, the operators in the functional calculi all form bounded operator algebras. For Lipschitz perturbations of the spheres in the complex plane, in the quaternionic space, in the Euclidean $n$-space, as well as perturbation of the $n$-complex sphere, and of the $n$-torus, analogous
$H^\infty$  functional calculi were established. They all correspond to the associated spherical Dirac differential operators,  see e.g. \cite{4} and \cite{5} and the references therein.

}
\end{remark}


\vskip 1cm
\begin{thebibliography}{10}



\bibitem{adler}
 S. Adler, {\em Quaternionic Quantum Field Theory}, Oxford University Press,
1995.


\bibitem{ADMcI}
D. Albrecht, X. Duong,  A. McIntosh,
{\em Operator theory and harmonic analysis}. Instructional Workshop on Analysis and Geometry, Part III (Canberra, 1995), 77--136,
Proc. Centre Math. Appl. Austral. Nat. Univ., 34, Austral. Nat. Univ., Canberra, (1996).


 \bibitem{FUCGEN}
D. Alpay, F. Colombo, J. Gantner, D. P. Kimsey,
{\em Functions of the infinitesimal generator of a strongly continuous quaternionic group}, preprint (2015),
 arXiv:1502.02954.

\bibitem{ACSBOOK}
D. Alpay, F. Colombo, I. Sabadini,
{\em Slice Hyperholomorphic Schur Analysis}, Preprint  (2015),
Quaderni del Dipartimento di Matematica, Politecnico di Milano,
QDD209.


 \bibitem{perturbation}
D.~{Alpay}, F.~{Colombo}, I.~{Sabadini},
 { \em Perturbation of the generator of a quaternionic evolution operator},
 Anal. Appl. (Singap.), {\bf 13} (2015), 347--370.



\bibitem{acgs}
D.~{Alpay}, F.~{Colombo}, J. Gantner, I.~{Sabadini},
{\em A new resolvent equation for the $S$-functional calculus},
 J. Geom. Anal., {\bf 25} (2015), 1939--1968.


\bibitem{ack}
D.~{Alpay}, F.~{Colombo},  D. P. Kimsey,
{\em The spectral theorem for for quaternionic unbounded normal operators based on the $S$-spectrum},
 {Preprint 2014,} available on arXiv:1409.7010.



\bibitem{acks2}
D.~{Alpay}, F.~{Colombo},  D. P. Kimsey, I.~{Sabadini}.
{\em The spectral theorem for unitary operators based on the $S$-spectrum},
 to appear Milan Journal of Mathematics (2016).





\bibitem{BvN}
G. Birkhoff,  J. von Neumann, {\em The logic of quantum mechanics},  Ann. of Math.,
 {\bf 37} (1936), 823-843.

\bibitem{bds} F. Brackx, R. Delanghe, F. Sommen, {\em Clifford Analysis},
Pitman Res. Notes in Math., 76, 1982.


\bibitem{CG}
 F. Colombo, J. Gantner,
{\em Fractional powers of quaternionic operators and Kato's formula using slice hyperholomorphicity},
preprint (2015), available on arXiv:1506.01266.



\bibitem{JGA} F. Colombo, I. Sabadini,
{\em On some properties of the quaternionic functional calculus},
J. Geom. Anal., {\bf 19}  (2009), 601-627.

\bibitem{CLOSED} F. Colombo, I. Sabadini,
{\em  On the  formulations of the quaternionic functional calculus},
 J. Geom. Phys., {\bf 60} (2010), 1490--1508.


\bibitem{MR2803786}
F. Colombo, I. Sabadini,
{\em  The quaternionic evolution operator},
 Adv. Math., {\bf 227} (2011), 1772--1805.

\bibitem{CAUCHY}
 F. Colombo, I. Sabadini, {\em The Cauchy formula with s-monogenic kernel and a functional calculus for noncommuting operators}. J. Math. Anal. Appl. {\bf 373} (2011),  655--679.

\bibitem{TREND}
F. Colombo, I. Sabadini, {\em A structure formula for slice monogenic functions and some of its consequences}, Hypercomplex analysis, 101--114, Trends Math., Birkh\"auser Verlag, Basel, 2009.

\bibitem{Global}
F. Colombo,  J. O. Gonzalez-Cervantes, I. Sabadini, {\em  A nonconstant coefficients differential operator associated to slice monogenic functions}, Trans. Amer. Math. Soc. {\bf 365} (2013), 303--318.

\bibitem{JFA}
 F. Colombo,  I. Sabadini, D. C. Struppa, {\em  A new functional calculus for noncommuting operators},
  J. Funct. Anal. {\bf 254} (2008),  2255--2274.

\bibitem{csss} F. Colombo, I. Sabadini, F. Sommen, D.C.
Struppa, {\em Analysis of Dirac Systems and Computational Algebra},
Progress in Mathematical Physics, Vol. 39, {Birkh\"auser}, Boston,
2004.

\bibitem{MR2752913}
F. Colombo, I. Sabadini,  D.~C. Struppa,
 {\em Noncommutative functional calculus. Theory and applications of slice regular functions},
 Vol. 289,  {\em Progress
  in Mathematics}.
 Birkh\"auser/Springer Basel AG, Basel, 2011.



\bibitem{ds}
N. Dunford, J. Schwartz. {\it Linear Operators, part I: General
Theory }, J. Wiley and Sons (1988).

\bibitem{3}
      G.I. Gaudry,  R. Long, T. Qian, {\em A martingale proof of L2 boundedness of Clifford-valued singular integrals}, Ann. Mat. Pura Appl. {\bf  165} (1993), 369--394.



\bibitem{GSSb}
 G. Gentili, C. Stoppato, D. C.  Struppa, {\em Regular functions of a quaternionic variable}.
  Springer Monographs in Mathematics. Springer, Heidelberg, 2013.



\bibitem{spectcomp} R. Ghiloni, V. Moretti, A. Perotti,
{\em Spectral properties of compact normal quaternionic operators},
 in Hypercomplex Analysis: New Perspectives and Applications
Trends in Mathematics, 133--143,  (2014).



\bibitem{GR} R. Ghiloni, V. Recupero,
{\em Semigroups over real alternative *-algebras: generation theorems and spherical sectorial operators},
Trans. Amer. Math. Soc.  ISSN 0002-9947 (In Press).



\bibitem{EngelNagel}
K.J. Engel, R. Nagel,
{\em  A short course on operator semigroups},
 Universitext. Springer, New York, (2006).


\bibitem{Haase}
M. Haase,  {\em The functional calculus for sectorial operators}. Operator Theory: Advances and Applications, 169. Birkhäuser Verlag, Basel, 2006.

\bibitem{Hille}
E. Hille, R. S. Phillips,  {\em Functional analysis and semi-groups},
 rev. ed. American Mathematical Society Colloquium Publications,
vol. 31. American Mathematical Society, Providence, R. I., (1957).

\bibitem{LIMC}
C. Li, A. McIntosh, {\em Clifford algebras and $H^\infty$ functional calculi of commuting operators}.
 Clifford algebras in analysis and related topics (Fayetteville, AR, 1993), 89–101, Stud. Adv. Math., CRC, Boca Raton, FL, 1996.

\bibitem{2LiMISe}
C. Li, A. McIntosh, S. Semmes,
{\em Convolution singular integrals on Lipschitz surfaces},
J. Amer. Math. Soc., {\bf 5} (1992),  455--481.


\bibitem{Kantorovitz}
S. Kantorovitz,
{\em Topics in operator semigroups},
Progress in Mathematics, 281. Birkhäuser Boston, Inc., Boston, MA, (2010).


\bibitem{jefferies} B. Jefferies, {\em Spectral properties of noncommuting operators},
Lecture Notes in Mathematics, 1843, Springer-Verlag, Berlin, 2004.


\bibitem{jmc} B. Jefferies, A. McIntosh, {\em The Weyl calculus and Clifford analysis},
Bull. Austral. Math. Soc., {\bf 57} (1998), 329--341.

\bibitem{jmcpw} B. Jefferies, A. McIntosh, J. Picton-Warlow, {\em The monogenic functional
calculus}, Studia Math., {\bf 136} (1999), 99--119.


\bibitem{mcp} A. McIntosh, A. Pryde, {\em A functional calculus for several commuting
operators}, Indiana U. Math. J., {\bf 36} (1987), 421--439.

\bibitem{LIATAO}
C. Li, A. McIntosh, T. Qian, {\em Clifford algebras, Fourier transforms and singular convolution operators on Lipschitz surfaces}, Rev. Mat. Iberoamericana, {\bf 10} (1994), 665--721.


\bibitem{Lunardi}
A. Lunardi, {\em Analytic semigroups and optimal regularity in parabolic problems},
 Progress in Nonlinear Differential Equations and their Applications, 16. Birkh\"auser Verlag, Basel, (1995).


\bibitem{McI1}
A. McIntosh,
{\em Operators which have an $H^\infty$ functional calculus}.
 Miniconference on operator theory and partial differential equations (North Ryde, 1986), 210--231,
Proc. Centre Math. Anal. Austral. Nat. Univ., 14, Austral. Nat. Univ., Canberra, (1986).

\bibitem{4}
	 T. Qian, {\em Singular integrals on star-shaped Lipschitz surfaces in the quaternionic space},
Math. Ann., {\bf 310} (1998),  601--630.

\bibitem{5}
 T. Qian, {\em Fourier analysis on starlike Lipschitz surfaces},
  J. Funct. Anal., {\bf 183} (2001), 370--412.

\bibitem{LW}
L. Weis, {\em  The $H^{\infty}$ holomorphic functional calculus for sectorial operators - a survey}.
 Partial differential equations and functional analysis, 263–294, Oper. Theory Adv. Appl., 168, Birkh\"auser, Basel, 2006.



\end{thebibliography}
\end{document}